\documentclass[11pt]{amsart}
\usepackage{amsmath,amsfonts,amssymb,amsthm}
\usepackage[alphabetic]{amsrefs}
\usepackage{hyperref}
\usepackage{tikz}
\usepackage{graphicx,color}
\usepackage{pictexwd,dcpic,epsf}
\usepackage{enumerate}

\newtheorem{theorem}{Theorem}[section]
\newtheorem{lemma}[theorem]{Lemma}
\newtheorem{prop}[theorem]{Proposition}
\newtheorem{question}[theorem]{Question}

\newtheorem{corollary}[theorem]{Corollary}

\theoremstyle{definition} 
\newtheorem{definition}[theorem]{Definition}

\theoremstyle{remark} 
\newtheorem{remark}[theorem]{Remark}
\usepackage{enumitem}

\newcommand{\ov}{\overline}

\newcommand{\angled}[1]{\langle#1\rangle}

\newcommand {\mr}{\mathrm}

\newcommand {\tX}{\widetilde X}

\newcommand{\act}{\curvearrowright}

\newcommand{\Xa}{X^{\ast}}
\newcommand{\Pa}{\Pi^{\ast}}
\newcommand{\Xb}{X^{\scalebox{0.53}{$\triangle$}}}

\newcommand{\ui}[1]{u_{#1}^{-1}}
\newcommand{\di}[1]{d_{#1}^{-1}}

\newcommand{\sk}[2]{#1^{(#2)}}
\newcommand {\bX}{\bar X}

\newcommand{\finv}{\bar f}

\newcommand{\Linv}{L^{-1}}

\newcommand{\length}{\operatorname{length}}

\begin{document}

\title[
Metric systolicity and two-dimensional Artin groups
]{
	Metric systolicity and two-dimensional Artin groups
}

\author{Jingyin Huang}
\address{Max Planck Institute for Mathematics, Vivatsgasse 7, 53111 Bonn, Germany}
\email{jingyin@mpim-bonn.mpg.de}

\author{Damian Osajda}
\address{Instytut Matematyczny,
Uniwersytet Wroc\l awski\\
pl.\ Grun\-wal\-dzki 2/4,
50--384 Wroc{\l}aw, Poland}
\address{Institute of Mathematics, Polish Academy of Sciences\\
\'Sniadeckich 8, 00-656 War\-sza\-wa, Poland}
\email{dosaj@math.uni.wroc.pl}

\subjclass[2010]{{20F65, 20F36, 20F67, 20F06, 20F10}} 
\keywords{two-dimensional Artin group, metrically systolic group, Word Problem, Conjugacy Problem}

\date{\today}

\begin{abstract}
	We introduce the notion of metrically systolic simplicial complexes. We study geometric and large-scale properties of such complexes and of groups acting on them geometrically. We show that all two-dimensional Artin groups act geometrically on metrically systolic complexes. As direct corollaries we obtain new results on two-dimensional Artin groups and all their finitely presented subgroups: we prove that  the Conjugacy Problem is solvable, and that the Dehn function is quadratic.
	We also show several large-scale features of finitely presented subgroups of two-dimensional Artin groups, lying background for further studies concerning their quasi-isometric rigidity.
\end{abstract}

\maketitle
\tableofcontents
\setcounter{tocdepth}{2}

\section{Introduction}
\label{s:intro}
Artin groups are among most intensively studied classes of groups in Geometric Group Theory. Conjecturally,
they possess nice geometric, topological, algebraic, and algorithmic properties, but most of such features 
are established only for rather restricted subclasses. Even in the case of two-dimensional Artin groups such basic questions as solvability of the Conjugacy Problem or the form of the Dehn function have remained open.
One, conjectural, approach to many questions concerning Artin groups is showing that they act geometrically
on CAT(0) spaces. Such results were established only for a number of rather limited subclasses of Artin groups, for: right-angled Artin groups (RAAGs) \cite{MR1368655}; certain classes of $2$--dimensional Artin groups \cite{brady2002two,BradyMcCammond2000}; Artin groups of finite type with three generators \cite{brady2000artin}; $3$--dimensional Artin groups of type FC \cite{bell2005three}; spherical Artin groups of type $A_4$ and $B_4$ \cite{brady2010braids}; $6$--strand braid group \cite{haettel20166}. 
Another method of treating Artin groups is finding other non-positive-curvature-like structures
describing them. Such approach was successfully carried out e.g.\ in \cite{AppelSchupp1983,Appel1984,Pride1986,Peifer,Bestvina1999}. In \cite{Artinsystolic} the authors 
undertake similar path showing that Artin groups of large type are systolic, that is, simplicially non-positively curved. This allowed to prove many new results about such groups.
In the current article we exhibit a non-positive-curvature-like structure of all two-dimensional
Artin groups and all their finitely presented subgroups, and conclude a number of new algorithmic, and large-scale geometric results
for those groups.
\medskip

As the main tool we introduce a new notion of \emph{metrically systolic} simplicial complex. Roughly speaking,
a simply connected flag simplicial complex with a piecewise Euclidean metric on its $2$--skeleton is {metrically systolic}
if all essential loops in links of vertices have (angle) length at least $2\pi$ (see Section~\ref{s:metric_syst} for details). 
This definition may be treated as a metric analogue of the definition of \emph{systolic complex} (see e.g.\ \cite{Chepoi2000,JanuszkiewiczSwiatkowski2006,Haglund,Artinsystolic}). 
Our main tool for exploring features of metrically systolic complexes is the use of disc diagrams.
It allows us to prove the following results about metrically systolic complexes and groups acting on them geometrically, that is, \emph{metrically systolic} groups.

\begin{theorem}
	\label{t:metsysprop}
	Let $X$ be a metrically systolic complex, and let $G$ be a metrically systolic group. Then the following properties hold.
	\begin{enumerate}
		\item Every loop in $X$ bounds a CAT(0) disc diagram (see Theorems~\ref{t:CAT(0)diagrem} and \ref{t:flatpt} in the text).
		\item The Dehn functions of $X$ and $G$ are quadratic (see Corollary~\ref{c:Dehn}).
		\item Finitely presented subgroups of $G$ are metrically systolic (see Theorem~\ref{t:fps}).
		\item If $G$ is torsion-free and $g^m$ is conjugated to $g^n$ only when $g^n=g^m$, for every $g\in G$, then the Conjugacy Problem is solvable in $G$ (see Theorem~\ref{t:cp}).
		\item $X$ and $G$ have constant filling radius for $2$--spherical cycles (see Theorem~\ref{t:sfrc} and Corollary~\ref{c:S2FRC2}).
		\item Morse Lemma for $2$--dimensional quasi-discs in $X$ (see Theorem~\ref{t:Morse}).
	\end{enumerate}
\end{theorem}
We believe that metrically systolic complexes deserve further extensive studies on their own; see a list of open questions in Section~\ref{s:open}. 
Geometrically, metric systolicity enables us to formalize a weaker notion of non-positively curved space where one only requires every minimal filling disc of a $1$--cycle to be non-positively curved. This naturally arises by examining the geometry of $2$--dimensional Artin groups. It is interesting to compare this with the work of Petrunin and Stadler \cite{PetruninStadler}, where (roughly speaking) they showed any minimal disc in a $CAT(0)$ space is $CAT(0)$. Thus it is natural to wonder whether one can set up this weaker notion in a more analytical way and apply it to natural classes of examples.

In the current
paper we focus on the use of metric systolicity in the context of Artin groups. To this end, starting with the standard
Cayley complex for a $2$--dimensional Artin group $G$, we modify it to obtain a metrically systolic $G$--complex. Therefore, we conclude the following.
\begin{theorem}[Theorem~\ref{t:Xmetrsys}]
	\label{t:main}
	Two-dimensional Artin groups are metrically systolic.
\end{theorem}
We refer to the next subsection for an intuitive explanation of the construction of the complex, as well as comparison with our previous work on constructing systolic complexes for large-type Artin groups from \cite{Artinsystolic}.

Direct consequences of Theorem~\ref{t:metsysprop} and Theorem~\ref{t:main} are new results on $2$--dimensional
Artin groups and their subgroups listed in Corollary~\ref{c:maincor}. Let us note that even if $2$--dimensional Artin groups were CAT(0),
this, a priori, would not say anything about their finitely presented subgroups -- this suggests an important 
advantage of metric systolicity. Moreover, by Brady and Crisp \cite{brady2002two}, there are $2$--dimensional Artin groups which can not act  nicely on $2$--dimensional CAT(0) complexes. On the other hand, metric systolicity enables us to stay in the $2$--dimensional world -- one need to study only CAT(0) disc diagrams. This will be convenient for our further work in \cite{Artinsrigidity} concerning quasi-isometries of $2$--dimensional Artin groups.

\begin{corollary}
	\label{c:maincor}
	Let $G$ be a finitely presented subgroup of a $2$--dimensional Artin group. Then:
	\begin{enumerate}
		\item $G$ has quadratic Dehn function and, in particular, solvable Word Problem;
		\item $G$ has solvable Conjugacy Problem;
		\item $G$ has constant filling radius for $2$--spherical cycles;
		\item Morse Lemma for two-dimensional quasi-discs in $G$ holds.
	\end{enumerate}
\end{corollary}
Dehn function, Word Problem, and Conjugacy Problem are among the most basic notions explored 
in the context of any group. Still, little was known about them for $2$--dimensional Artin groups and their 
finitely presented subgroups prior to our work.

As far as we know there have been no general results concerning Dehn function for $2$--dimensional Artin groups before. Chermak \cite{Chermak1998} proved the Word Problem is solvable for $2$--dimensional Artin groups, but no general statement of this type have been known for finitely presented subgroups.

Solvability of the Conjugacy Problem for $2$--dimensional Artin groups and their finitely presented subgroups follows directly from Theorem~\ref{t:metsysprop} (4). It is so because $2$--dimensional Artin groups are torsion-free by \cite{CharneyDavis}, and their cyclic subgroups are undistorted (see Theorem~\ref{t:further} below).
Prior to our result solvability of the Conjugacy Problem was established only for a few particular subclasses
of Artin groups: braid groups \cite{Garside1969}, finite type Artin groups \cite{BrieskornSaito1972,Deligne1972,Charney1992,MR1314589}, large-type Artin groups \cite{Appel1984,AppelSchupp1983}, triangle-free Artin groups \cite{Pride1986}, $3$--dimensional Artin groups of type FC \cite{bell2005three}, certain $2$--dimensional Artin groups with 3 generators \cite{brady2002two}, some Artin groups of Euclidean type \cite{MR2150887,MR2208796,MR2985512,MR3351966,mccammond2011artin}, RAAG's \cite{MR1023285,MR1285550,MR1314099,MR2546582}. In particular, the question about solvability of the Conjugacy Problem has been open for the class of $2$--dimensional Artin groups.

Assertions (3) and (4) from Corollary~\ref{c:maincor} could be derived without referring to
metric systolicity. However, for the proof of the strong form of (3), as presented in Theorem~\ref{t:sfrc} in the text, the use of metric systolicity is very convenient. This result, in turn, is a crucial ingredient in the proof of the Morse Lemma for two-dimensional quasi-discs (see the proof of Theorem~\ref{t:Morse}). The latter
is an important large-scale feature of metrically systolic complexes, groups, and of $2$--dimensional Artin groups.

The metrically systolic complexes constructed in Theorem~\ref{t:main}, as well as the large-scale features mentioned above, will play fundamental role in the study of quasi-isometric invariants of $2$--dimensional Artin groups in our subsequent work \cite{Artinsrigidity}. 
For applications in \cite{Artinsrigidity} we need another result, presented in the following theorem.
It does not rely on metric systolicity, and follows from known facts, but it seems that it is not
present in the literature.
\begin{theorem}[Theorem~\ref{t:absbgrp} and Corollary~\ref{c:solv}]
	\label{t:further}
	Let $A_{\Gamma}$ be a $2$--dimensional Artin group. Then
	\begin{enumerate}
		\item every abelian subgroup of $A_{\Gamma}$ is quasi-isometrically embedded;
		\item nontrivial solvable subgroups are either $\mathbb Z$ or virtually $\mathbb Z^2$.
	\end{enumerate}
\end{theorem}

\noindent
{\bf Comments on the proof of Theorem~\ref{t:main}.}
Here we present a rough idea of the construction of metrically systolic complexes for two dimensional Artin groups. 

Let $\Gamma$ be a finite simple graph with each of its edges labeled by an integer $\ge 2$. An \emph{Artin group with defining graph $\Gamma$}, denoted $A_\Gamma$, is given by the following presentation. The generators of $A_\Gamma$ are in one to one correspondence with vertices of $\Gamma$, and there is a relation of the form $\underbrace{aba\cdots}_{n}=\underbrace{bab\cdots}_{n}$ whenever two vertices $a$ and $b$ are connected by an edge labeled by $n$.

An Artin group is of \emph{dimension $d$} if it has cohomological dimension $d$. By Charney and Davis \cite{CharneyDavis}, $A_\Gamma$ has dimension $\le 2$ if and only if for any triangle $\Delta\subset\Gamma$ with its sides labeled by $p,q,r$, we have $\frac{1}{p}+\frac{1}{q}+\frac{1}{r}\le 1$. In particular, the class of all \emph{large-type} Artin groups, where the label of each edge in $\Gamma$ is $\ge 3$, is properly contained in the class of Artin groups of dimension $\le 2$.

Let  $A_\Gamma$ be an Artin group of dimension $\le 2$ and let $\Xa_\Gamma$ be the presentation complex of $A_\Gamma$. A natural way to metrize $\Xa_\Gamma$ is to declare each $2$--cell in $\Xa_\Gamma$ is a regular polygon in the Euclidean plane. However, if we take $2$--cells $\Pi_1$ and $\Pi_2$ (say, two $n$--gons) such that $P=\Pi_1\cap\Pi_2$ is a path with $\ge 2$ edges, then any interior vertex of $P$ is not non-positively curved. Let $o_i$ be the center of $\Pi_i$ and let the two endpoints of $P$ be $v_1$ and $v_2$. Let $K$ be the region in $\Pi_1\cup \Pi_2$ bounded by the $4$--gon whose vertices are $o_1$, $o_2$, $v_1$ and $v_2$. Those positively curved corner points are contained in $K$. Now we add a new edge $e$ between $o_1$ and $o_2$ and add two new triangles $\{\Delta_i\}_{i=1}^2$ such that the three sides of $\Delta_i$ are $e$, $\overline{o_1v_i}$ and $\overline{o_2v_i}$; see Figure~\ref{f:explanation}.

\begin{figure}[h]
	\centering
	\includegraphics[width=0.5\linewidth]{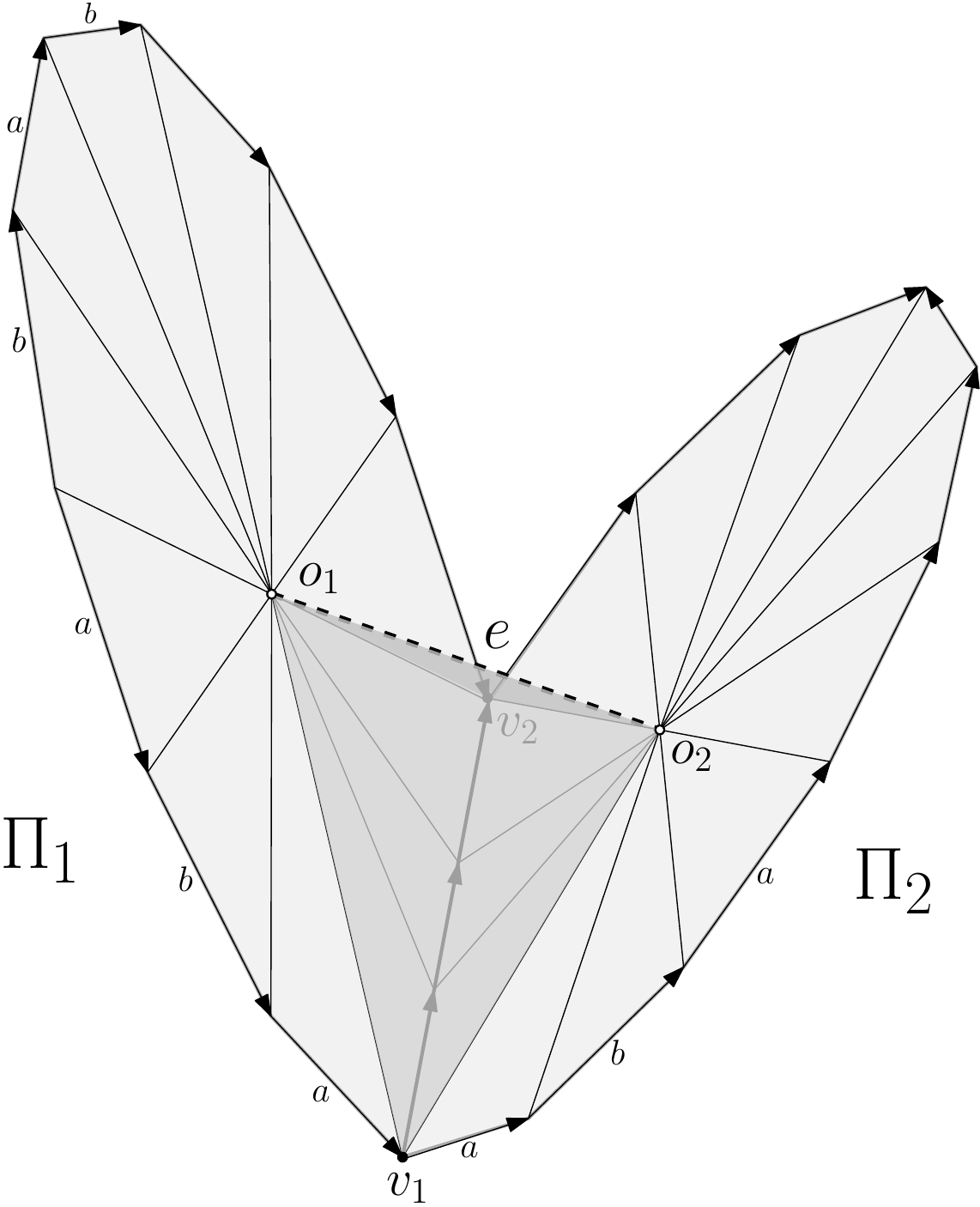}
	\caption{Adding the edge $e$ and the triangles $\Delta_1=o_1o_2v_1$, and $\Delta_2=o_1o_2v_2$ (in the universal cover of $\Xa_\Gamma$).}
	\label{f:explanation}
\end{figure}

Geometrically, one can think of $K$ as a configuration sitting inside the Euclidean $3$--space $\mathbb E^3$. Then positively curved points in $K$ give rise to corners in the configuration. Now we use the polyhedron bounded by $K\cup\Delta_1\cup\Delta_2$ to fill these corners. Combinatorially, one can think of $\Delta_1\cup\Delta_2$ as a replacement of $K$ to get rid of positively curved points in the disc diagram.

Now we decide the length of $e$. From the geometric viewpoint, $e$ should be shorter if $P$ is longer. From the combinatorial viewpoint, we would like $o_i$ to be flat after we replace $K$ by $\Delta_1\cup\Delta_2$. Thus $\angle_{o_1}(v_1,o_2)=\angle_{o_1}(o_2,v_2)=\frac{|P|}{4n}2\pi$ ($|P|$ is the number of edges in $P$), which determines the length of $e$. 

Pick a triangle $\Delta\subset\Gamma$, then $\Delta$ gives rise to three $2$--cells arranged in a cyclic fashion around a vertex $v$. The condition on two dimensionality of $A_\Gamma$ implies $v$ is already non-positively curved in such configuration, so we do not apply any modifications here.

The main difference between the construction in \cite{Artinsystolic} and the one in this paper is that the former is purely combinatorial, while the current one uses both the metric and combinatorial structure. Thus the method in this paper has more flexibility and applies to a much larger class of Artin groups. Moreover, the structure of flat points in the disc diagrams is more convenient for our later use in \cite{Artinsrigidity}. However, since we are now outside the purely combinatorial setting, 
some results from \cite{Artinsystolic} -- e.g.\ biautomaticity -- are much harder to obtain.

\medskip

\noindent
{\bf Organization of the paper.}
In Section~\ref{s:metric_syst} we define metrically systolic complexes and prove their fundamental property -- every cycle can be filled by a $CAT(0)$ disc diagram. In Section~\ref{s:msprop}, we prove the rest of properties in Theorem~\ref{t:metsysprop}, using $CAT(0)$ disc diagrams as a basic tool. In Section~\ref{s:dihedral}, we construct the metrically systolic complexes for dihedral Artin groups. In Section~\ref{s:link}, we study the local properties of these complexes with two purposes. First we show the complexes in Section~\ref{s:dihedral} are indeed metrically systolic. Second we show that there are no local obstructions to metric systolicity if we glue these complexes together under certain conditions. In Section~\ref{s:general} we glue the complexes for dihedral Artin groups to construct the metrically systolic complexes for any $2$--dimensional Artin groups, and prove Theorem~\ref{t:main}. In Section~\ref{s:open}, we prove Theorem~\ref{t:further} and leave some open questions about metrically systolic groups and complexes.
\medskip

\noindent
{\bf Acknowledgments.} 
The authors were partially supported by (Polish) Narodowe Centrum Nauki, grants no.\ UMO-2015/\-18/\-M/\-ST1/\-00050 and UMO-2017/25/B/ST1/01335. A major part of the work on the paper was carried out while D.O.\ was visiting McGill University.
We would like to thank the Department of Mathematics and Statistics of McGill University
for its hospitality during that stay.

\section{Metrically systolic complexes}
\label{s:metric_syst}
In this section we introduce the notion of metrically systolic complex. Then we show its most
important feature, used later extensively for proving other properties of metrically systolic complexes and groups. The feature is the existence of CAT(0) disc diagrams filling any cycle inside the complex; see Theorem~\ref{t:CAT(0)diagrem} and Theorem~\ref{t:flatpt}. 
The proofs presented in Subsection~\ref{s:msCAT(0)} go by modifying any given disc diagram
to a CAT(0) one by performing a finite sequence of local ``moves". 
As an immediate consequence
we obtain the quadratic Dehn function in Corollary~\ref{c:Dehn}.

\subsection{Definition}
\label{s:msdef}
Let $X$ be a flag simplicial complex with its two-skeleton $\sk{X}{2}$ equipped with a metric $d$
in which every $2$--simplex (triangle) is isometric to a Euclidean triangle.
For a vertex $v\in X$ its \emph{link}, denoted $lk(v,X^{(1)})$, is the full subcomplex (subgraph) of $\sk{X}{1}$
spanned by all vertices adjacent to $v$. Every link is equipped with an \emph{angular metric}, defined as follows. For an edge $\overline{v_1v_2}$, we define the \emph{angular length} of this edge to be the angle $\angle_v(v_1,v_2)$ with the apex $v$. This turns the link into a metric graph, and the angular metric, which we denote by $d_\angle$, is the path metric of this metric graph (note that a priori we do not know whether $\angle_v(v_1,v_2)=d_\angle(v_1,v_2)$ for adjacent vertices $v_1$ and $v_2$). The \emph{angular length} of a path $\omega$ in the link, which we denote by $\length_\angle(\omega)$, is the summation of angular lengths of edges in this path. In this paper we assume that the following weak form of triangle inequality
holds for angular length in $X$: for each $v\in X$ and every three pairwise adjacent vertices $v_1,v_2,v_3$
in the link of $v$ we have that $\angle_v(v_1,v_3)\le  \angle_v(v_1,v_2)+\angle_v(v_2,v_3)$.
Then we call $X$ (with metric $d$) a \emph{metric simplicial complex}.

\begin{remark}
	Note that we allow that the inequality becomes equality -- intuitively it corresponds to degenerate $2$--simplices
	in a link, which corresponds to degenerate $3$--simplices in $X$.
\end{remark}

For $k=4,5,6,\ldots$, a simple $k$--cycle $C$ in a simplicial complex is \emph{$2$--full} if 
there is no edge connecting any two vertices in $C$ having a common neighbor in $C$. 

\begin{definition}[Metrically systolic complexes and groups]
	\label{d:metric_syst}
	A link in a metric simplicial complex is \emph{$2\pi$--large} if every $2$--full
	simple cycle in the link has angular length at least $2\pi$.
	A metric simplicial complex $X$ is \emph{locally $2\pi$--large} if every its link is $2\pi$--large.
	A simply connected locally $2\pi$--large metric complex is called a \emph{metrically systolic} complex.
	\emph{Metrically systolic} groups are groups acting geometrically by isometries on metrically systolic complexes.
\end{definition}

\begin{remark}
	A \emph{systolic complex}, that is, a connected simply connected flag simplicial complex
	for which all full cycles in links consist of at least six edges is metrically systolic
	when equipped with the metric in which all triangles are Euclidean triangles with edges of unit lengths.
	For more on systolic complexes see e.g.\ \cite{Chepoi2000,JanuszkiewiczSwiatkowski2006,Haglund,JanuszkiewiczSwiatkowski2007,Wise2003-sixtolic,Elsner2009-flats,ChepoiOsajda,Artinsystolic}.
\end{remark} 

\subsection{CAT(0) disc diagrams}
\label{s:msCAT(0)}
A standard reference for singular disc diagrams (or van Kampen diagrams) is \cite[Chapter V]{LSbook}. Our approach is close to the ones from
\cite[Section 5]{Chepoi2000} and \cite[Section 1]{JanuszkiewiczSwiatkowski2006}. The material is rather standard, however we need a precise description of diagram modifications for further use.

A \emph{singular disc} $D$ is a simplicial complex isomorphic to a finite connected and simply connected subcomplex
of a triangulation of the plane. There is the (obvious) \emph{boundary cycle} for $D$, that is, a map
from a triangulation of $1$--sphere (circle) to the boundary of $D$, which is injective on edges. 
For a cycle $C$ in a simplicial complex $X$, a \emph{singular disc diagram for $C$} is a simplicial
map $f\colon D \to X$ from a singular disc $D$, which maps the boundary cycle of $D$ onto $C$; see Figure~\ref{f:diagrams} (left).  
\begin{figure}[ht!]
	\centering
	\includegraphics[width=1\textwidth]{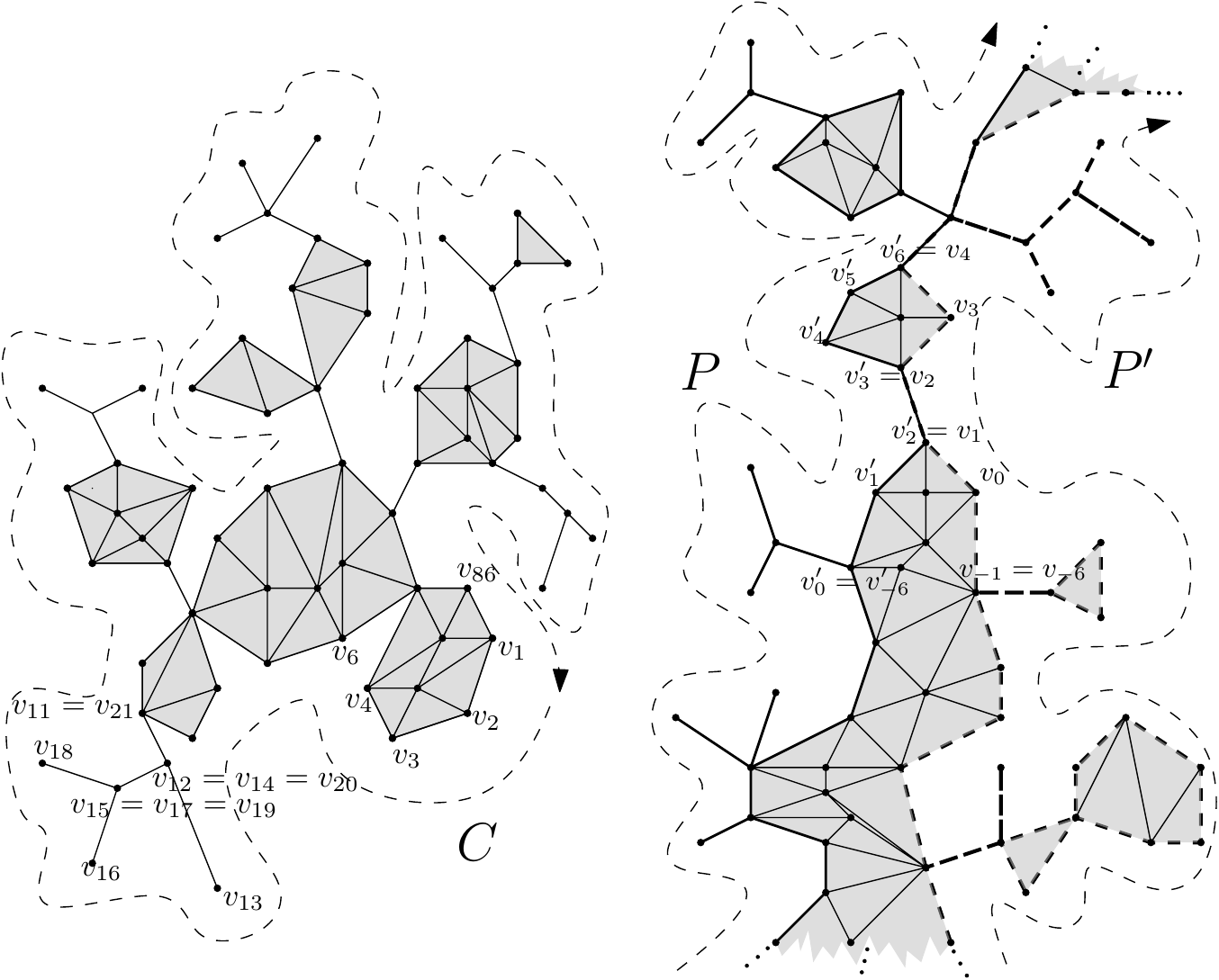}
	\caption{A singular disc with the boundary cycle $C=(v_1,v_2,\ldots,v_{86})$ (left), and a singular strip  for a pair $P=(\ldots,v_{-1},v_0,v_1,v_2,\ldots)$ (solid thick), $P'=(\ldots,v'_{-1},v'_0,v'_1,v'_2,\ldots)$ (dashed thick).}
	\label{f:diagrams}
	\end{figure}
By the relative simplicial approximation theorem \cite{Zeeman1964}, for every cycle in a simply connected simplicial complex there exists a singular disc diagram (cf.\ also van Kampen's lemma e.g.\ in \cite[pp.\ 150-151]{LSbook}). Below we describe how to obtain singular disc diagrams with some additional properties, by modifying a given one.

A singular disc diagram is called \emph{nondegenerate} if it is injective on all simplices.
It is \emph{reduced} if distinct adjacent triangles (i.e., triangles sharing an edge) are mapped 
into distinct triangles. The \emph{area} of a singular disc diagram is the number of $2$--simplices (triangles) in the underlying singular disc. A singular disc diagram for a cycle $C$ in $X$ is \emph{minimal}
if it has the minimal area among singular disc diagrams for $C$ in $X$.
For a metric simplicial complex $X$ and a nondegenerate singular disc diagram $f\colon D \to X$ we equip
$D$ with a metric in which $f|_{\sigma}$ is an isometry onto its image, for every simplex $\sigma$ in $D$.
Then, $f$ is a \emph{CAT(0) singular disc diagram} if $D$ is CAT(0), that is, if the angular length of
every link in $D$ being a cycle (that is, the link of an interior vertex in $D$) is at least $2\pi$. 

Parallelly to singular disc diagrams one may consider a related notion of singular strip diagrams.
A \emph{singular strip} $S$ is 
a simplicial complex isomorphic to an infinite connected and simply connected subcomplex
of a triangulation of the plane whose complement has two infinite components.
The two infinite paths being boundaries of those components are
called the \emph{boundary paths} of $S$. Having two infinite paths $P,P'$ in $X$, a \emph{singular strip diagram for the pair $P,P'$} is a simplicial map $f\colon S \to X$ from a singular strip $S$ into $X$ mapping boundary paths of $S$ onto, respectively, $P$ and $P'$; see Figure~\ref{f:diagrams} (right). A \emph{nondegenerate, reduced} or \emph{CAT(0)
singular strip diagram} is defined analogously as the corresponding singular disc diagram.   
\medskip

Having a singular disc diagram $f\colon D \to X$ for a cycle $C$ in $X$ we describe a way of producing
a new singular disc diagram $f'\colon D'\to X$ for $C$, with some additional properties (see e.g.\ Theorem~\ref{t:CAT(0)diagrem} below). In order to do this we need elementary operations -- \emph{moves} --
described below. 
\medskip

\noindent
{\bf A-move: }Assume there exist pairwise adjacent vertices $u,v,w$ not bounding a triangle in $D$, that
is, there are vertices $v_1,\ldots,v_k$ in the region in $D$ bounded by edges between $u,v,w$. The new singular
disc $D'$ is obtained from $D$ by removing all the vertices $v_i$ (and hence also edges containing them); 
see Figure~\ref{f:moves} (at the top).
The new map $f'\colon D' \to X$ is defined as $f'(x)=f(x)$, for all vertices $x$ in $D'$, and then extended simplicially. Such modification is called the \emph{A-move on $u,v,w$} and is denoted by A($u,v,w$).
\begin{figure}[ht!]
	\centering
	\includegraphics[width=0.9\textwidth]{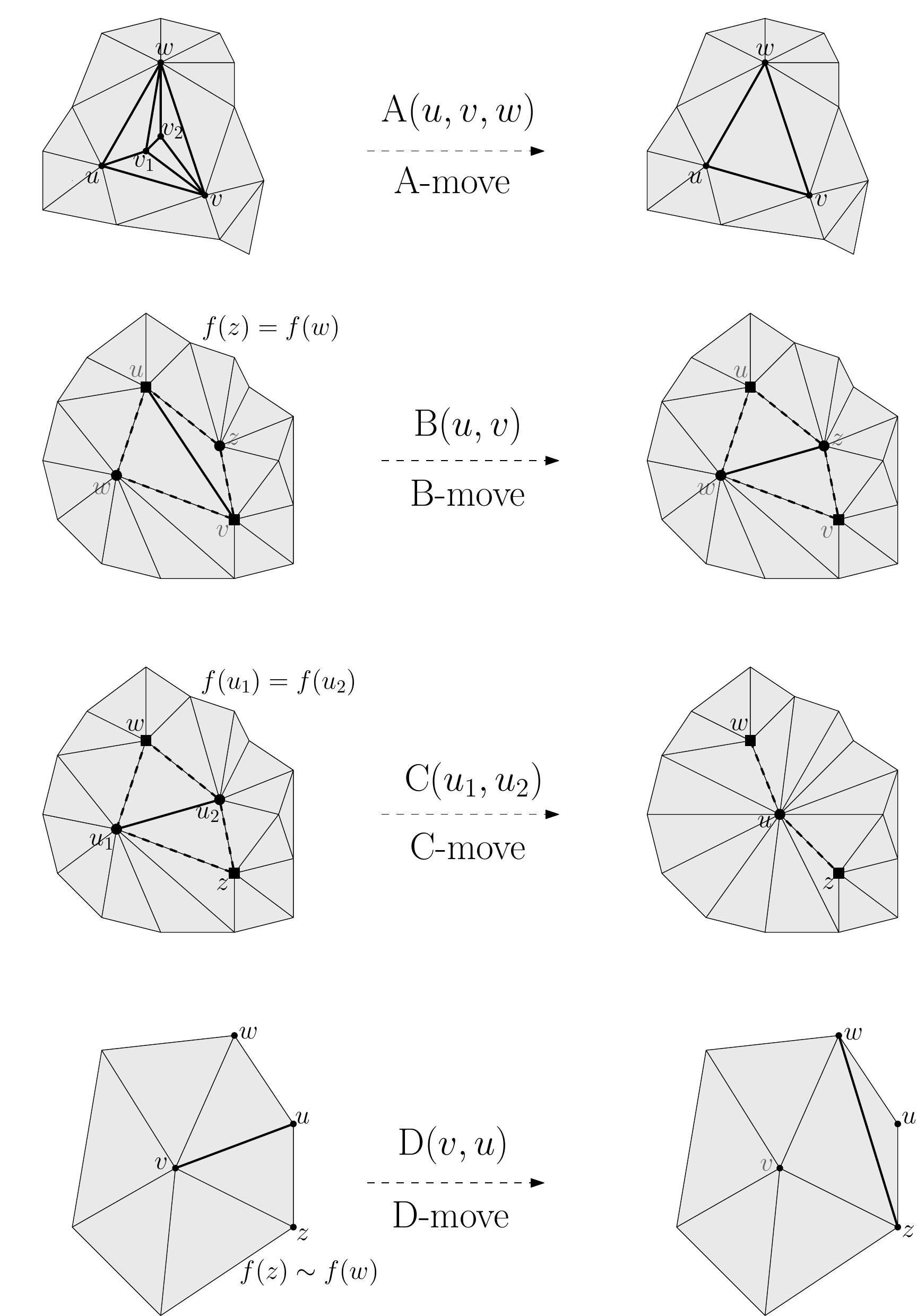}
	\caption{Moves.}
	\label{f:moves}
\end{figure}
\medskip

For the next moves we assume that the situation as above does not happen, that is, each triple of pairwise adjacent
vertices defines a triangle in $D$. In particular it means that for each internal edge $\ov{uv}$ in $D$
there are exactly two vertices $w,z$ each adjacent to both $u$ and $v$. 

\medskip

\noindent
{\bf B-move: }Assume there are 
two triangles $uvw$ and $uvz$ such that $f(w)=f(z)$. 
The new singular disc $D'$ is obtained
from $D$ by removing the edge $\ov{uv}$ and adding an edge $\ov{wz}$; see Figure~\ref{f:moves}.
By our assumptions $D'$ is a simplicial singular disc.
The new map $f'\colon D' \to X$ is defined as $f'(x):=f(x)$, for all vertices $x$ in $D'$, and then 
extended simplicially. Such modification is called \emph{B-move on $u,v$} and is denoted by B($u,v$). 

\medskip

\noindent
{\bf C-move: }Assume there is an edge $\ov{u_1u_2}$ such that $f(u_1)=f(u_2)$. Such edge need to be internal, so that there are two triangles
$u_1u_2w$ and $u_1u_2z$ containing the edge.
The new singular disc $D'$ is obtained
from $D$ by removing $u_1,u_2$ (and all edges containing them), and then adding a new vertex $u$
adjacent to all vertices (of $D$ except $u_1,u_2$) that are adjacent in $D$ to $u_1$ or $u_2$; see Figure~\ref{f:moves}.
By our assumptions $D'$ is a simplicial singular disc.
The new map $f'\colon D' \to X$ is defined as $f'(x):=f(x)$, for all vertices $x\neq u$ in $D'$, and $f'(u):=f(u_1)=f(u_2)$, and then 
extended simplicially. Such modification is called \emph{C-move on $u_1,u_2$} and is denoted by C($u_1,u_2$).
\medskip

\noindent
{\bf D-move: }Assume there is a vertex $v$ in $D$ with the link being a cycle (that $v$ is an internal vertex), and such that for a vertex $u$ adjacent to $w,z$ in the link the vertices $f(w)$ and $f(z)$ are adjacent (we write $f(w)\sim f(z)$). Then the new singular diagram $D'$ is obtained
from $D$ by removing the edge $\ov{uv}$ and adding an edge $\ov{wz}$; see Figure~\ref{f:moves} (bottom).
The new map $f'\colon D' \to X$ is defined as $f'(x):=f(x)$, for all vertices $x$ in $D'$, and then 
extended simplicially. Such modification is called \emph{D-move on $v,u$} and is denoted by D($v,u$). 
\medskip

The following lemma is essentially the same as \cite[Lemma 5.1]{Chepoi2000} and  
\cite[Lemma 1.6]{JanuszkiewiczSwiatkowski2006}. Although in the latter two only simple cycles are considered,
the general case follows by decomposing a given cycle into simple pieces. We omit the straightforward proof.

\begin{lemma}
	\label{l:reddiag}
	Let $f\colon D \to X$ be a singular disc diagram for a cycle $C$ in a simplicial complex $X$. Then by applying A-moves, B-moves, and C-moves the diagram may be modified to a nondegenerate reduced singular disc diagram for $C$.
	In particular, any minimal singular disc diagram for $C$ is nondegenerate and reduced.
\end{lemma}

The main technical tool for dealing with metrically systolic complexes are CAT(0) singular disc diagrams.
Their existence is established in the following theorem. It is an analogue of a result for systolic complexes
obtained in
\cite[pp.\ 159--161]{Chepoi2000} and \cite[Lemma 1.7]{JanuszkiewiczSwiatkowski2006}. The proof is also an analogue of the systolic case proof.
Before the theorem we prove a useful lemma.

\begin{lemma}
	\label{l:fillshort}
	Let $f\colon D\to X$ be a singular disc diagram into a metrically systolic complex $X$. Suppose that 
	there is an interior vertex $v$ in $D$ whose link is a cycle $C$ of angular length less than $2\pi$. Then, by performing
	a finite number of
	A-, and D-moves we may find a singular disc diagram $f'\colon D'\to X$ such that 
	$D'$ is a union of the full subcomplex of $D$ spanned by all vertices of $D$ except $v$, and triangles with vertices in $C$, and the map $f'$ agrees with $f$ on all vertices of $D'$ and on all edges
	coming from $D$. 
\end{lemma}
\begin{proof}
	We proceed by induction on the combinatorial length of $C$. If this length is $3$ then we perform
	A-move. Assume that $C$ consists of at least $4$ edges. 
	Denote $C=(v_1,v_2,\ldots,v_k)$. 
	Then $\ov f(C)=(\ov f(v_1),\ov f(v_2), \ldots, \ov f(v_k))$ is a cycle in $X$ of angular length 
	less then $2\pi$. There is $2<l\le  k$ such that $C'=(\ov f(v_1),\ov f(v_2), \ldots, \ov f(v_l))$
	is a simple cycle. This is a cycle in the link of $\ov f(v)$ of angular length less than $2\pi$.
	If $l=3$ then $\ov f(v_1)$ and $\ov f(v_3)$ are adjacent.
	If $l>3$ then, by metric systolicity, $C'$ is not $2$--full. This means that
	there exists a vertex, say $\ov f(v_2)$, such that its neighbors in $C'$ -- in our case
	$\ov f(v_1)$ and $\ov f(v_3)$ -- are adjacent. Hence we may perform D-move D($v,v_2$), to obtain
	a new singular disc diagram $\ov f'\colon \ov D' \to X$. Furthermore, $\angle_v(v_1,v_3)$ in $\ov D'$ is at most $\angle_v(v_1,v_2)+\angle_v(v_2,v_3)$ in $\ov D$, so that the angular length of the link of 
	$v$ in $\ov D'$, being the cycle $(v_1,v_3,\ldots,v_k)$, is less than $2\pi$.
	By the inductive assumption we obtain the desired diagram $f'\colon D'\to X$.
\end{proof}

\begin{theorem}[CAT(0) disc diagram]
	\label{t:CAT(0)diagrem}
	Let $f\colon D\to X$ be a singular disc diagram for a cycle $C$ in a metrically systolic
	complex $X$. By performing a finite number of A-, B-, C-, D-moves the diagram may be modified to
	a CAT(0) nondegenerate reduced singular disc diagram $f' \colon D' \to X$ for $C$. 
	Furthermore:
	\begin{enumerate}
		\item 
		$f'$ does not use any new vertices in the sense that there is an injective map $i$ from the vertex set of $D'$ to the vertex set of $D$ such that $f=f'\circ i$ on the vertex set of $D$;
		\item the number of $2$--simplices in $D'$ is at most the number of $2$--simplices in $D$; 
		\item any minimal singular disc diagram for $C$ is CAT(0) nondegenerate and reduced.
	\end{enumerate}
\end{theorem}
\begin{proof}
	We proceed with successive diagrams $\ov f \colon \ov D \to X$, starting from $\ov f:= f$ depending on the following cases.
	\medskip
	
	\noindent
	{\bf Case 1: }A-move, B-move, or C-move may be performed. Then the new diagram $\ov f' \colon \ov D' \to X$ is obtained by
	performing the corresponding move.
	
	\medskip
	
	\noindent
	{\bf Case 2: }No A-move, B-move, or C-move may be performed and there exists an internal vertex $v$
	whose link is a cycle $C=(v_1,v_2,\ldots,v_k)$ of angular length less than $2\pi$ (in the metric induced from $X$). Then, by Lemma~\ref{l:fillshort} there exists a singular disc diagram 
		$\ov f'\colon \ov D' \to X$, where $\ov D'$ is obtained from $\ov D$ by replacing the star of $v$ with a disc without internal vertices, and $\ov f'$ coincides with $\ov f$ on all vertices except $v$.
	\medskip

	\noindent
	{\bf Case 3: } We are not in situations from Case 1 or Case 2. Then the diagram $\ov f \colon \ov D \to X$
	is a CAT(0) nondegenerate reduced singular disc diagram for $C$.
	\medskip
	
	After performing modifications as in Case 2, the area of the diagram decreases.
	Proceeding as in Case 1, that is performing A-moves, B-moves, or C-moves eventually decreases
	the area of the diagram. It is so because A-move and C-move decrease the area, and after performing
	B-move we are in position to perform A-move or C-move. Hence eventually we end up in Case 3.
	
	Assertions (1), (2) and (3) follow immediately from the construction.
\end{proof}

\begin{corollary}
	\label{c:Dehn}
	The Dehn function of a metrically systolic complex or group is at most quadratic.
\end{corollary}

For further applications (e.g.\ in \cite{Artinsrigidity}) we will need singular disc diagrams with some
further features (see Theorem~\ref{t:flatpt} below).
To construct them we have to consider other types of moves: E-moves and F-moves described below.
Again, starting from a singular disc diagram $f\colon D\to X$ into a metrically systolic complex $X$ we
construct a new diagram $f'\colon D'\to X$.
For the new moves we assume that we are in the situation when no A-, B-, or C-move may be performed, and
there is an interior vertex $v$ and two vertices $w,z$ in its link such that $f(w)=f(z)$.
Observe that then $w$ and $z$ are not adjacent.
\medskip 

\noindent
{\bf E-move: }
Assume that there does not exist a vertex different than $v$ and adjacent to both $w,z$.
We assume furthermore that the angular lengths of two paths between $w$ and $z$ in the link of $v$ are strictly smaller
than $2\pi$.
The new disc diagram $f'\colon D'\to X$ is obtained as follows. First we construct an intermediate singular disc $D''$ by ``collapsing" vertices $w,z$ to a single vertex $x$, that is, we remove $w,z$, and introduce a
new vertex $x$ adjacent to all vertices that were adjacent in $D$ to $w$ or $z$; see Figure~\ref{f:EFmoves} (top). 
Furthermore, we add two ``copies" $v',v''$ of the vertex $v$, adjacent to vertices in two paths of the link of $v$, and to $x$. A singular disc diagram $f''\colon D'' \to X$ is defined by setting $f''(x)=f(w)$, 
$f''(v')=f''(v'')=f(v)$, and $f''$ agrees with $f$ otherwise. Observe that the angular lengths of
links of $v'$ and $v''$ are strictly smaller than $2\pi$. Hence, by double application of Lemma~\ref{l:fillshort} we find a desired singular disc diagram $f'\colon D'\to X$ with the two links
filled without internal vertices.  
\medskip

\noindent
{\bf F-move: }Assume that there exists a vertex $u$ different than $v$ and adjacent to both $w,z$.
We first construct a singular disc diagram $f''\colon D'' \to X$ by joining $w$ and $z$ by an edge,
removing edges from $v$ ``crossing" the new edge $\ov{wz}$ and adding
a copy $v'$ of $v$ adjacent to vertices in the original link of $v$ not adjacent to $v$ anymore; see Figure~\ref{f:EFmoves} (bottom). In $D''$ there is a triangle $wzu$, and performing the A-move $A(u,w,z)$
we obtain the desired singular disc diagram $f'\colon D' \to X$.

\medskip

\begin{figure}[ht!]
	\centering
	\includegraphics[width=0.8\textwidth]{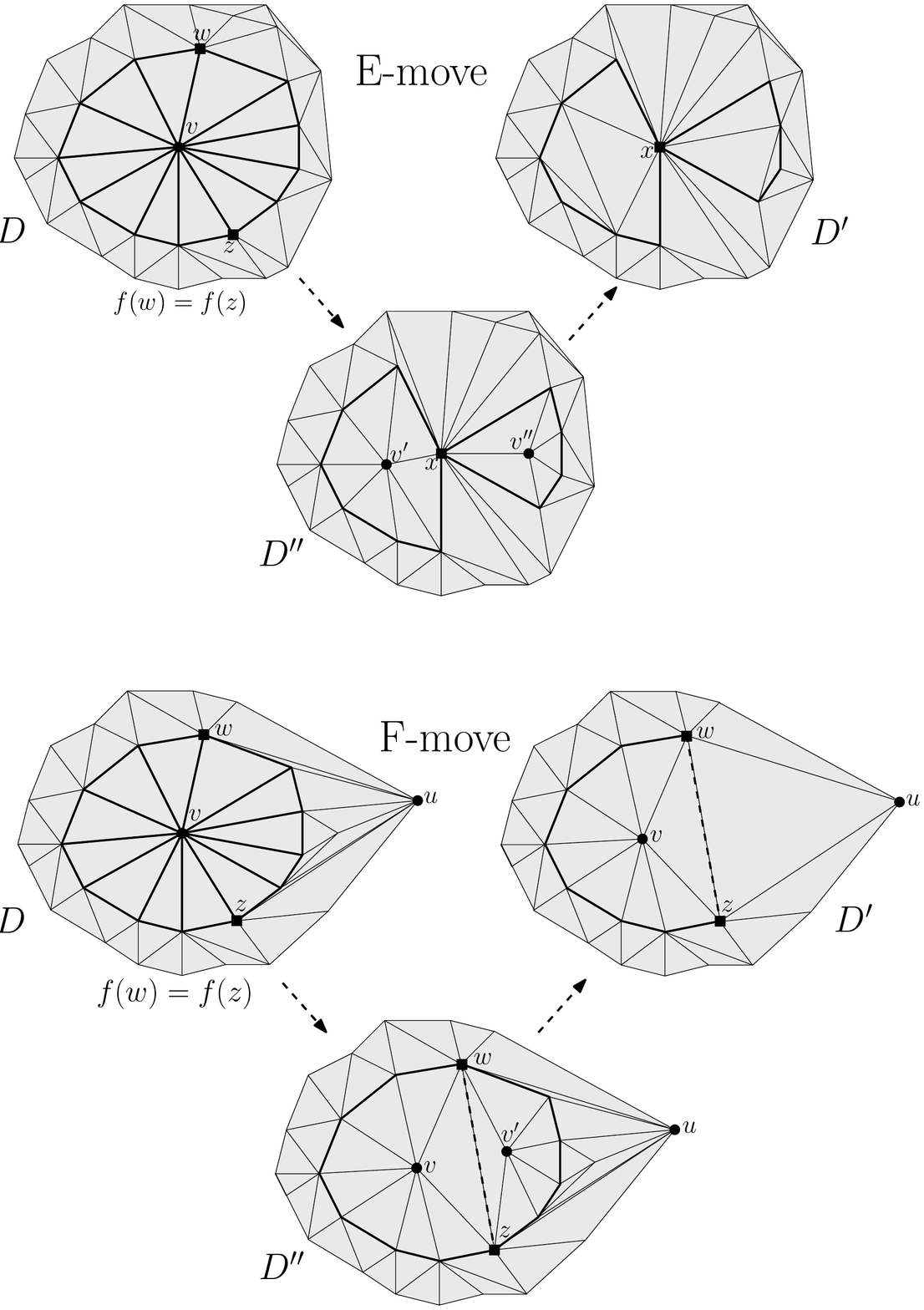}
	\caption{E-move and F-move.}
	\label{f:EFmoves}
\end{figure}
\medskip

\begin{theorem}[CAT(0) disc diagram II]
	\label{t:flatpt}
	Let $f\colon D\to X$ be a singular disc diagram for a cycle $C$ in a metrically systolic
	complex $X$. By performing a finite number of A-, B-, C-, D-, E-, F-moves the diagram may be modified to
	a CAT(0) nondegenerate reduced singular disc diagram $f' \colon D' \to X$ for $C$ satisfying the following
	property. For every flat vertex $v\in D'$ the restriction $f|_{\mr{St}(v)}$ is injective. 
	Furthermore:
	\begin{enumerate}
		\item 
		$f'$ does not use any new vertices in the sense that there is an injective map $i$ from the vertex set of $D'$ to the vertex set of $D$ such that $f=f'\circ i$ on the vertex set of $D$;
		\item the number of $2$--simplices in $D'$ is at most the number of $2$--simplices in $D$; 
		\item any minimal singular disc diagram for $C$ is such.
	\end{enumerate}
\end{theorem}
\begin{proof}
	By Theorem~\ref{t:CAT(0)diagrem}, using finitely many A-, B-, C-, D-moves we may
	modify $f$ to a CAT(0) nondegenerate reduced singular disc diagram $\ov f$. Moreover,
	we may reach the situation when no A-, B-, C-move is possible. If for every
	flat vertex $v$ the restriction $\ov f|_{\mr{St}(v)}$ is injective then we are done with $f'=\ov f$.
	If not, we are in a position to perform an E-move or an F-move.
	Both decrease the area.
	
	Applying iteratively the above procedure we finally obtain the desired singular disc diagram $f'\colon D'\to X$. Assertions (1), (2), and (3) follow directly from the construction.
\end{proof}

\begin{remark}
	Observe that the assertion of the lemma is not true if the vertex is not flat -- the star of such vertex
	could be mapped onto the simplicial cone over a wedge of two cycles.
\end{remark}

\begin{remark}
	We could reduce the number of moves for proving Theorem~\ref{t:CAT(0)diagrem} or Theorem~\ref{t:flatpt} by allowing 
	singular discs to be non-simplicial, as e.g.\ in \cite[proof of Lemma 1.6]{JanuszkiewiczSwiatkowski2006}.
	We decided to stay in the realm of simplicial complexes.
\end{remark}

\section{Properties of metrically systolic complexes and groups}
\label{s:msprop}
In this section we prove several properties of metrically systolic complexes and groups.
In particular, such properties hold for two-dimensional Artin groups, and -- as explained 
in Subsection~\ref{s:msfps} below -- for all their finitely presented subgroups.

\subsection{Finitely presented subgroups}
\label{s:msfps}
In this subsection we show that being metrically systolic for groups is inherited by taking finitely presented subgroups. It follows that all subsequent features (and the quadratic isoperimetric inequality established above) of metrically systolic groups are
valid also for all their finitely presented subgroups. In particular, they hold for all finitely
presented subgroups of two-dimensional Artin groups.

\begin{theorem}
	\label{t:fps}
	Finitely presented subgroups of metrically systolic groups are metrically systolic.
\end{theorem}
\begin{proof}
	In view of \cite[Theorem 1.1]{HanlonMartinez} (compare also \cite[Corollary 5.8]{Wise2003-sixtolic}) it is enough to show that the class of locally $2\pi$--large
	complexes is closed under taking covers and full subcomplexes.
	
	Let $\tX \to X$ be a cover of a locally $2\pi$--large complex $X$. Then links in $\tX$ are combinatorially isomorphic to
	links in $X$. It follows that such links equipped with a metric induced by the isomorphism are
	$2\pi$--large. Such metric on links is the angular metric coming from the metric on
	$\tX$ induced by the covering. Therefore, $\tX$ is metrically systolic.
	
	Let $\bX$ be a full subcomplex of a metrically systolic complex $X$, equipped with a subcomplex metric. Let $C$ be a $2$--full simple cycle in the link of a vertex of $\bX$. By fullness of $\bX$, $C$ is $2$--full in $X$, hence its angular length is at least $2\pi$. Therefore, the angular length of $C$
	in $\bX$ is at least $2\pi$ as well. It follows that $\bX$ is locally $2\pi$--large.
\end{proof}

\subsection{Solvability of the Conjugacy Problem}
\label{s:mscp}
In this subsection we show that the Conjugacy Problem is solvable for torsion-free metrically systolic 
groups satisfying some additional technical assumption; see Theorem~\ref{t:cp}.
The proof is a typical argument for showing solvability of the Conjugacy Problem
in the non-positive curvature setting; see e.g.\ \cite[pp.\ 445-446]{BridsonHaefliger1999}

Below, and in further parts of the article we use the following convention concerning quasi-isometries. 
\begin{definition}
	Assume $K,L>1$.
	A \emph{$(K,L)$--quasi-isometric embedding} is a map $f\colon (X,d_X)\to (Y,d_Y)$ between metric spaces such that 
	\begin{align*}
	K^{-1} d_X(x,y) -L \le  d_Y(f(x),f(y)) \le  K d_X(x,y)+L,
	\end{align*}
	for all $x,y\in X$.
	
	A \emph{$(K,L)$--quasi isometry} is a $(K,L)$--quasi-isometric embedding $f\colon X\to Y$ having an $L$--coarse inverse $\bar f \colon Y\to X$, that is, a $(K,L)$--quasi-isometric embedding
	such that $d_X(x,\bar f \circ f(x))\le  L$ for all $x\in X$, and
	$d_Y(y,f \circ \bar f(y))\le  L$ for all $y\in Y$.
\end{definition}

For the rest of the subsection let $G$ be a torsion-free group acting geometrically on a metrically
systolic complex $X$. We will use here the induced metric $d$ in the one-skeleton of
$X$. By scaling the metric we may assume that all edges have length at most $1$.
Let $S$ be a finite (symmetrized) generating set for  $G$, and let $\Gamma:= \mr{Cay}(G, S)$ be the corresponding
Cayley graph. Let $d_S$ be the word metric on $G$ and (the
$0$--skeleton of) $\Gamma$, and $|g|_S=d_S(1_G,g)$.

The following two lemmas are standard but we formulate them for the purpose of refereeing 
to constants appearing later.
 The first one is just the Milnor-Schwarz lemma.
\begin{lemma}
	\label{l:constOI}
	There exist $K_1,L_1>1$ such that for every vertex $x\in X$ the orbit map   $\Phi \colon (G,d_S) \to X: g\mapsto gx$ is a $(K_1,L_1)$--quasi-isometry, and for every vertex $v\in X$ there exists $g\in G$ such that 
	$d(v,\Phi(g))\le  L_1$.
\end{lemma}

Let $D$ be a planar CAT(0) $2$--complex constructed from triangles isometric to triangles
in $X$. Let $\delta$ be a CAT(0) geodesic between two given vertices $v,u$ in $D$.
A path $\delta'$ in the $1$--skeleton of $D$ is \emph{approximating} the geodesic 
$\delta$ if $\delta'$ is contained in the union of all edges and triangles intersecting $\delta$,
and $\delta'$ is the shortest path with this property. 
The following is a consequence of e.g.\ \cite[Proposition I.7.31]{BridsonHaefliger1999}.

\begin{lemma}
	\label{l:constAPPROX}
	There exist constants $K_2,L_2>1$ depending only on the geometry of $X$ (in fact, on the set of isometry types of triangles in $X$) such that $K_2^{-1}|\delta'|-L_2\le  |\delta| \le  K_2|\delta'|+L_2$. 
\end{lemma}

Let $K:=\max \{K_1,K_2\}$ and $L=\max \{L_1,L_2\}$. In particular, it means that 
the assertions of Lemmas~\ref{l:constOI} and \ref{l:constAPPROX} hold when 
the corresponding constants $K_i$ and $L_i$ are replaced by $K$ and $L$.

\begin{lemma}
	\label{l:CP}
	Let $g,h\in G$ be conjugate elements, such that, for every vertex $v\in X$, the shortest path between 
	$v$ and $gv$ consists of at least $4$ edges.
	Then there exists an element $p\in G$, conjugating them, that is, $g=php^{-1}$, and such that
	$|p|_S\le A$, where $A$ is a constant depending only on $|g|_S$ and $|h|_S$ (and on the action of $G$ on $X$).
\end{lemma}
\begin{proof}
	For every generator $s\in S$, choose a geodesic $1$--skeleton path $q(s)$ in $X$,
	between $x$ and $sx$.
	Let $p$ be an element conjugating $g$ and $h$. We will show that starting with $p$
	we may find a conjugator $p'$ with  
	$|p'|_S\le A$, where $A$ is a constant depending only on $|g|_S$ and $|h|_S$.
	
	Let $\alpha_S, \gamma_S$, and $\eta_S$ be geodesics in $\Gamma$ between $1_G$ and,
	respectively $p,g$, and $h$.  Let $s^p_1\cdots s^p_a$, $s^g_1\cdots s^g_b$, and $s^h_1\cdots s^h_c$
	be words in $S$ defined by these geodesics.
	Let $\alpha$ be the concatenation of paths $q(s^p_1),s^p_1q(s^p_2),s^p_1s^p_2q(s^p_3)\ldots,$\\$s^p_1\cdots s^p_{a-1}q(s^p_a)$.
	Similarly, let $\gamma$ be the concatenation of paths $q(s^g_1),s^g_1q(s^g_2),$\\$s^g_1s^g_2q(s^g_3)\ldots,s^g_1\cdots s^g_{b-1}q(s^g_b)$, and let $\eta$ be the concatenation of paths $q(s^h_1),$\\$s^h_1q(s^h_2),\ldots,s^h_1\cdots s^h_{c-1}q(s^h_c)$.
	Consider the cycle $C$ based at $x$, being the concatenation of (in this order) $\gamma,g\alpha,p\eta$, and $\alpha$; see Figure~\ref{f:CP}. 
	
	By Lemma~\ref{l:constOI}, there exist constants $E_1$ and $F_1$ depending only on $|g|_S$ and $|h|_S$
	(and the action of $G$ on $X$) such that $|\gamma|\le  E_1$, $|\eta|\le  F_1$,
	where $|\cdot|$ denotes the $d$--length. In what follows we will consider constants depending on
	$E_1,F_1$, and $K,L$ leading, eventually, to a constant $A$ as in the statement of the lemma.
	 
	Let $f\colon D\to X$ be a singular disc diagram for the cycle $C$. We create a singular strip diagram
	$\ov f \colon \ov D \to X$ as follows. For every $n\in \mathbb Z$ let $D^n$ be a copy of $D$,
	and let $f^n$ be the simplicial map such that $f^n(v):=g^nf(v)$, for every vertex $v\in D$ -- here
	we identify $D^n$ with $D$. In particular $f^0=f$. Next, for every $n$, we identify the copy of the path $g\alpha$ in $D^n$ with the copy of the path $\alpha$ in $D^{n+1}$. This way we obtain a singular
	strip $\ov D = \bigcup_{n\in \mathbb Z} D^n$. We define the map $\ov f$ as the union of maps
	$f^n$, for all $n$. This way we obtain the singular strip diagram $\ov f \colon \ov D \to X$ for the pair 
	of paths $\ov{\gamma}$, $\ov {p\eta}$, where $\ov{\gamma}$ is the concatenation of paths $g^n\gamma$, and 
	$\ov{p\eta}$ is the concatenation of paths $g^np\eta$, for all $n\in \mathbb Z$; see Figure~\ref{f:CP}.
		\begin{figure}[ht!]
			\centering
			\includegraphics[width=0.5\textwidth]{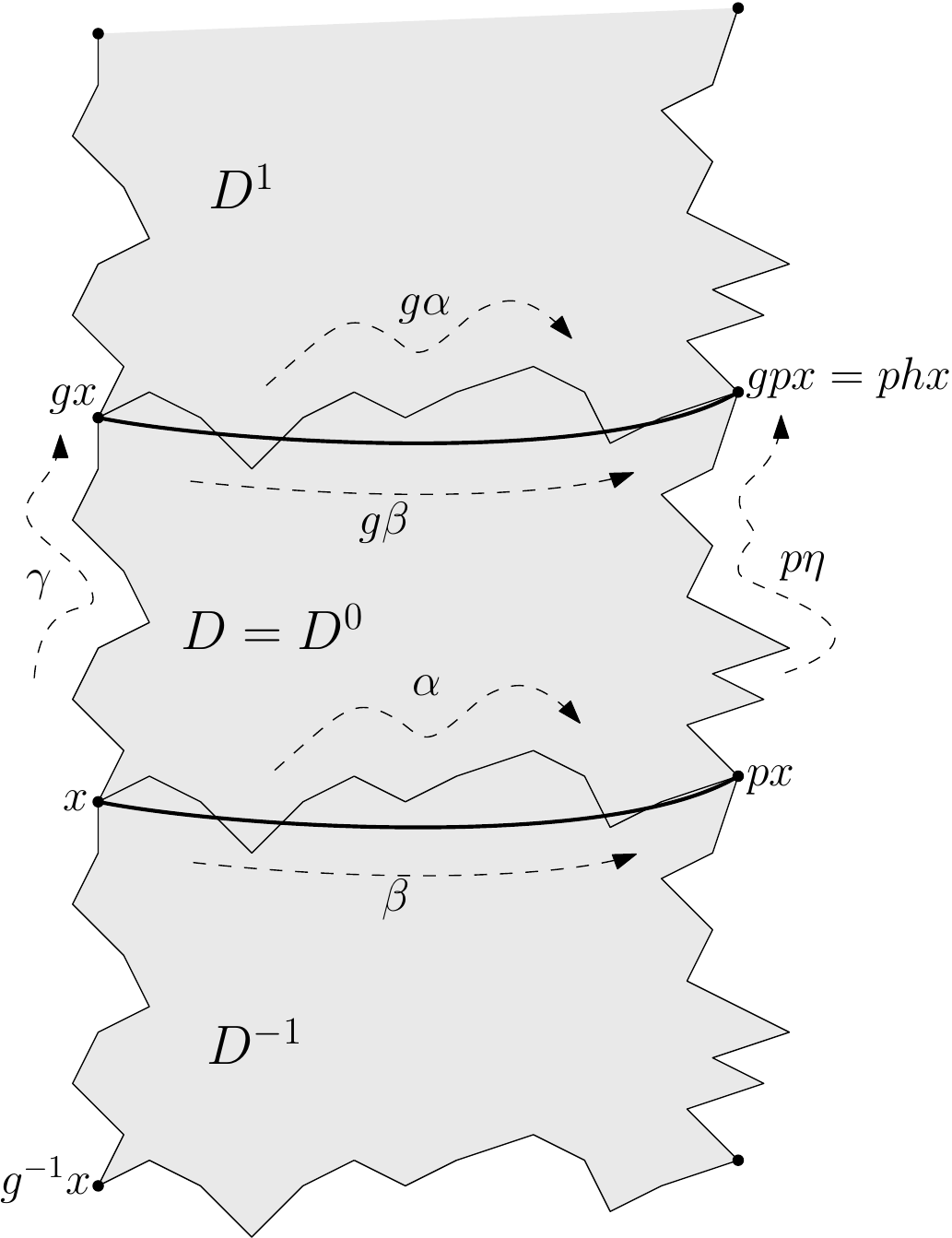}
			\caption{Scheme for proving Lemma~\ref{l:CP}.}
			\label{f:CP}
		\end{figure}	 
	Observe that there is a $\langle g \rangle$--action on $\ov D$: $g^nD^m=D^{n+m}$, and that the 
	map $\ov f$ is equivariant with respect to this action and the $\langle g \rangle$--action
	on $X$.
	
	For every $m\neq n$, and for each triple of pairwise adjacent vertices $v_1,v_2,v_3$ in $\ov D$, the A-moves A($g^mv_1,g^mv_2,g^mv_3$) and A($g^nv_1,g^nv_2,g^nv_3$) may be performed independently, since 
	the shortest path between $g^mv_i$ and $g^nv_j$ has at least $3$ edges.
	Similarly, B-moves, C-moves, and D-moves may be performed independently for distinct translates
	of the defining vertices. Thus, we may define an \emph{equivariant A-move on $u,v,w$} as
	the modification consisting of A-moves A($g^nu,g^nv,g^nw$), for all $n$. Similarly we define
	\emph{equivariant B-move}, \emph{equivariant C-move}, and \emph{equivariant D-move}.
	As an equivariant analogue of Theorem~\ref{t:CAT(0)diagrem} we claim that by performing a finite
	number of equivariant moves the singular strip diagram $\ov f \colon \ov D \to X$
	may be modified to a CAT(0) nondegenerate reduced singular strip diagram $\ov f' \colon \ov D' \to X$ for the pair $\ov{\gamma},\ov{p\eta}$.
	
	Let $\beta'$ be the CAT(0) geodesic in $\ov D'$ with endpoints $x,px$ (that is, their preimages in $\ov D'$). Let $d_{\ov D'}$ denote the CAT(0) distance in ${\ov D'}$. Since $d_{\ov D'}(x,gx)$\\$\le 
	|\gamma|\le  E_1$, and $d_{\ov D'}(x,hx)\le 
	|\eta|\le  F_1$, by the CAT(0) geometry, and the $\langle g \rangle$--invariance of $\ov D'$, we
	have 
	\begin{align*}
	d_{\ov D'}(y,gy)\le  \max \{E_1,F_1\}=:E,
	\end{align*}
	for every point $y\in \beta'$.

	Let $\beta=(v_0=x,v_1,\ldots,v_r=px)$ be a path in the $1$--skeleton of $\ov D'$ with endpoints $x,px$ (that is, their preimages in $\ov D'$) approximating the CAT(0) geodesic in $\ov D'$ between $x$ and $px$.
	Then $g\beta$ approximates $g\beta'$, and hence, for every vertex $v$ of $\beta$, we have
	\begin{align}
	\label{e:20}
	d_{\ov D'}(v,gv)\le  d_{\ov D'}(v,y)+d_{\ov D'}(y,gy)+d_{\ov D'}(gy,gv)\le 
	1+E+1,
	\end{align}
	where $y\in \beta'$ is a point closest to $v$.
		Using Lemma~\ref{l:constAPPROX} we get
	\begin{align}
	\label{e:30}
	d(v,gv)\le  K(E+2)+L.
	\end{align}
	For every $i\in \{1,\ldots,r-1\}$ we find $g_i\in G$ such that $d(v_i,g_ix)\le  L$ (see Lemma~\ref{l:constOI}). Additionally, we set $g_0=1_G$ and $g_r=p$.
	Then, by Lemma~\ref{l:constOI}, for every $i=0,1,\ldots,r-1$ we have
	\begin{align}\label{e:40}
		\begin{split}
			d_S(g_i,g_{i+1})&=d_S(gg_i,gg_{i+1})\le  Kd(g_ix,g_{i+1}x)+KL \le  \\
			&\le  K(d(g_ix,v_i)+d(v_i,v_{i+1})+d(v_{i+1},g_{i+1}x))+KL\le  \\
			&\le  K(L+1+L)+KL=:\ov L.
		\end{split}
	\end{align}
	By (\ref{e:20}) and (\ref{e:30}), we have
	\begin{align}
	\label{e:50}
		\begin{split}
		d_S(g_i,gg_i)&\le  Kd(g_ix,gg_{i}x)+KL \le  \\
		&\le  K(d(g_ix,v_i)+d(v_i,gv_i)+d(gv_i,gg_{i}x))+KL \le  \\
		&\le  K(L+E+2+L)+KL=:\widehat L.
		\end{split}
	\end{align}
	For every $i$ we choose a $d_S$--geodesic between $g_i$ and $g_{i+1}$. Let $\beta_S$ be their
	concatenation. This is a path in $\Gamma$ connecting $1_G$ and $p$.
	By (\ref{e:40}) and (\ref{e:50}), for every $a\in \beta_S$, we have 
	\begin{align*}
		\begin{split}
			d_S(a,ga)\le  d_S(a,g_i)+d_S(g_i,gg_i)+d_S(gg_i,ga)\le  2\ov L+ \widehat L=: \widetilde{L},
		\end{split}
	\end{align*} 
	where $g_i$ is the closest to $a$ among $g_i$'s.
	
	Now consider the quadrilateral $Q$ in $\Gamma$ formed by paths $\gamma_S, \beta_S, p\eta_S, g\beta_S$.
	For every vertex $v$ on $\beta_S$ pick a geodesic $\gamma_v$ between $v\in \beta_S$ and $gv\in g\beta_S$.
	There are at most $A:=|S|^{\widetilde{L}+1}$ different up to $G$--action on $\Gamma$ paths of length
	less than $\widetilde{L}$. Hence if $|\beta_S|_S>A$ then there are two vertices $v,v'\in \beta_S$
	such that the two paths $\gamma_v$ and $\gamma_{v'}$ are the same up to $G$.
	Cutting $Q$ along such paths and gluing together we obtain a quadrilateral $Q'$ formed by paths 
	$\gamma_S, \beta_S', p\eta_S, g\beta_S'$, and such that again $d_S(a,ga)\le  \widetilde{L}$,
	for all $a\in \beta_S'$. This way we construct a quadrilateral $Q''$ consisting
	of paths $\gamma_S, \beta_S'', p\eta_S, g\beta_S''$, with $|\beta_S''|_S\le  A$.
	Hence we obtain an element $p'\in G$ conjugating $g$ and $h$, with $|p'|_S\le  A$.
\end{proof}

\begin{theorem}
	\label{t:cp}
	Let $G$ be a torsion-free metrically systolic group such that for every element $g\neq 1_G$ of $G$
	if $g^n$ and $g^m$ are conjugated then $n=m$. Then 
	the Conjugacy Problem is solvable for $G$.
\end{theorem}
\begin{proof}
	Suppose $G$ acts geometrically on a metrically systolic complex $X$. Let $g=php^{-1}$. By the assumption on conjugates of $g$, we may find $n$ such that the displacement of $g^n$ is as large as in Lemma~\ref{l:CP}. Note that $n$ does not depend on $g$, it only depends on the number of elements in the orbit of $G$ contained in a ball of $X$ of given size. Clearly $g^n=ph^np^{-1}$. By Lemma~\ref{l:CP} the displacement of $p$ is bounded by
	value depending only on displacements of $g$, and $h$, and the action of $G$ on $X$. Hence there is 
	a bound on the number of possible $p$'s. Note that this number is of the same order as the number of words we need to search in the $CAT(0)$ case.
\end{proof}

\subsection{Spherical fillings}
\label{s:sfrc}

The following result is a direct analogue of \cite[Theorem 9.2]{JanuszkiewiczSwiatkowski2007} and \cite[Theorem 2.4]{Elsner2009-flats} concerning systolic complexes.
\begin{theorem}
	\label{t:sfrc}
	Let $X$ be a metrically systolic complex and $f\colon S \to X$ be a simplicial map from a triangulation of the two-sphere. Then $f$ can be extended to a simplicial map $F\colon B \to X$, where
	$B$ is a triangulation of a $3$-–ball such that $\partial B=S$ and $B$ has no internal vertices.
\end{theorem}
\begin{proof}
	The proof is a direct analogue of the proof of \cite[Theorem 2.4]{Elsner2009-flats}. 
	It goes by the induction on the area (number of triangles) of $S$. If the area is $4$ (the smallest possible) then $S$ is the $2$--skeleton of the tetrahedron and the result follows by flagness of $X$.
	For larger area we consider the two following subcases.
	\medskip
	
	\noindent
	{\bf Case 1: }$S$ is not flag. Then we proceed exactly as in the proof of \cite[Theorem 2.4]{Elsner2009-flats}: we decompose $S$ into two discs along an ``empty" triangle, create two spheres of smaller area
	and use the induction assumption.
	\medskip
	
	\noindent
	{\bf Case 2: }$S$ is flag.	Since the $2$--sphere does not admit a metric of non-positive curvature
	there exists a vertex $v$ in $S$ whose link, a cycle $C$, has angular length less than $2\pi$.
	We have the decomposition $S=D_1\cup D_2$, where $D_1$ is the star of $v$ and $D_2$ is the complement 
	of the interior of $D_1$. By Lemma~\ref{l:fillshort} the cycle $f|_C$ has 
	a singular disc diagram $D$ with no internal vertices.
	Let $B_1$ be the simplicial cone over $D$ with apex $v$, and let $F_1\colon B_1 \to X$ be the simplicial map such that $F_1(u)=f(u)$, for all vertices $u$ (it is well defined by flagness of $X$).
	Then $S_2=D_2\cup D$ is a simplicial sphere of area smaller than the one of $S$. Let $f_2\colon S_2\to X$ 
	be the simplicial map coinciding on vertices with $f$. Applying the inductive assumption we extend it to 
	$F_2\colon B_2\to X$, where $B_2$ is a triangulation of the ball with no internal vertices satisfying
	$\partial B_2=S_2$. Finally we put $B=B_1\cup B_2$ and $F=F_1\cup F_2$.
\end{proof}

Januszkiewicz-\'Swi\c atkowski introduced in \cite{JanuszkiewiczSwiatkowski2007} the notion
of \emph{constant filling radius for $k$--spherical cycles}, shortly \emph{$S^k$FRC}.
This is a coarse feature of metric spaces saying, roughly, that in large scale every $k$--sphere
has a filling within its uniform neighbourhood. A direct consequence of Theorem~\ref{t:sfrc} is the following.  

\begin{corollary}
	\label{c:S2FRC2}
	Metrically systolic complexes and groups are $S^2$FRC, that is, they have constant filling radius for $2$--spherical cycles.  
\end{corollary}

\subsection{Morse Lemma for $2$--dimensional quasi-discs}
\label{s:Morse}
In this subsection we prove a Morse Lemma  for $2$--dimensional quasi-discs. It states, roughly speaking,
that, for a given cycle $C$ in a metrically systolic complex, a quasi-isometrically embedded disc diagram is contained in an  $a$--neighbourhood of any other singular disc diagram for $C$, with $a$ independent of the size of the disc.

We use the combinatorial metric on simplicial complexes. In particular, the distance between adjacent vertices is $1$.
Let $B(R,v)$ denote the \emph{(combinatorial) ball} of radius $R$ centered at $v$, that is the 
full subcomplex of a simplicial complex spanned by all vertices at distance at most $R$ from $v$. Similarly, the \emph{sphere} $S(R,v)$ is the full subcomplex spanned by all vertices at distance $R$ from $v$. Let
$T(r,R;v)$ denote the \emph{tube} (annulus) of radii $r,R$ around $v$, that is, the full
subcomplex spanned by all vertices $u$ such that $r\le  d(v,u)\le  R$.
Observe that for $(L,A)$-quasi-isometry $f$ we have $f(T(r,R;v)\subseteq T(\Linv r -A,LR+A;f(v))$.
Recall that the \emph{systolic plane}, denoted $\mathbb E_{\triangle}^2$, is the triangulation of the Euclidean plane by regular triangles.
\begin{theorem}[Morse Lemma for $2$--dimensional quasi-discs]
	\label{t:Morse}
	Let $D$ be a combinatorial ball in the systolic plane $\mathbb E_{\triangle}^2$. Let
	$f\colon D \to X$ be a disc diagram for a cycle $C$ in $X$ 
	being an $(L,A)$--quasi-isometric embedding. 
	Let $g\colon D' \to X$ be a singular disc diagram for $C$.
	Then $\mr{im}(f)\subseteq N_a(\mr{im}(g))$, where $a>0$ is a constant depending only on $L$ and $A$.
\end{theorem}
\begin{proof}
	There exist constants $L'\ge  L$ and $A'\ge  A$  depending only on $L,A$ such that 
	$f\colon D\to f(D)$ is an $(L',A')$--quasi-isometry, and there is an $(L',A')$--quasi-isometry
	$\finv \colon f(D)\to D^{(0)}$ such that $\finv \circ f$ and $f\circ \finv$ are $A'$--close to
	identities. Let $K\ge  \max \{L',A',3\}$. We will further work with $K$ instead of $L,A$ --
	this will make the computations easier. In particular $(L',A')$--quasi-isometries are
	$(K,K)$--quasi-isometries.
	We claim that $a=K^{20}$ satisfies the assertion of the 
	lemma.

	We proceed by contradiction. Suppose there is a vertex 
	$v\in \mr{im}(f)\setminus N_a(\mr{im}(g))$. Then clearly $d(v,C)>a$.
	Let $v':=f(\finv(v))$. Then $d(v,v')\le  K$.

	Let $X_1=\mr{span}(f(D)\cap B(K^{12};v))$, and let $X_2=\mr{span}(g(D')\cup(f(D)-B(K^8,v))$.

	\medskip
	
	Let $\alpha=(v_0,v_1,\ldots,v_k)$ be a cycle in $S(K^{10},\finv(v))$ being a generator of
	$H_1(S(K^{10},\finv(v));\mathbb Z)$. 
	Observe that then $\alpha$ represents also a generator of $H_1(T(K^5,K^{15},\finv(v)
	);\mathbb Z)$.
	Let $f(\alpha)=(f(v_0),f(v_1),\ldots,f(v_k))$ be the cycle (possibly with $f(v_i)=f(v_j)$ for some $i\neq j$) being the image of $\alpha$.
	Observe that, by $d(v,v')\le  K$ and $K\ge 3$, we have
	\begin{align*}
	\label{e:b10}
	\begin{split}
	f(\alpha)&\subseteq T(K^{-1}K^{10}-K,K^{11}+K,v')\subseteq \\ & \subseteq T(K^{-1}K^{10}-2K,K^{11}+2K,v)
	\subseteq T(K^8,K^{12};v).
	\end{split}
	\end{align*}	
	\medskip
	
	\noindent
	{\bf Claim.} The cycle $f(\alpha)$ is not null-homologous inside $T(K^8,K^{12};v)\cap X_1$.
	\medskip
	
	To prove the claim suppose, by contradiction, that $f(\alpha)$ is null-homologous in $T(K^8,K^{12};v)\cap X_1$.
	Then there exists a simplicial map 
	\begin{align*}
	h\colon T \to T(K^8,K^{12};v)\cap X_1
	\end{align*}
	from a simplicial $2$--complex $T$ to 
	$T(K^8,K^{12};v)\cap X_1$ sending the boundary cycle to $f(\alpha)$. 
	We define a map $\finv\circ h\colon T \to  D$ as follows.
	For every vertex $u\in T$ we send it to $\finv\circ h(u)$. An edge $\ov{uw}$ is sent 
	to a geodesic between $\finv(u)$ and $\finv(w)$. A triangle $uwz$ is sent to a singular disc in $D$ bounded
	by the chosen geodesic between images of vertices.	
	Since
	\begin{align*}
	\finv(T(K^8,K^{12};v)\cap X_1)&\subseteq T(K^7-K,K^{13}+K,\finv(v))\subseteq \\
	&\subseteq T(K^6,K^{14},\finv(v)),
	\end{align*}
	and since the image of every edge has diameter at most $K$, and similarly the image of every 
	triangle has diameter at most $K$, we have that the image of $\finv\circ h$ is contained in
	$T(K^6-K,K^{14}+K,\finv(v))$.
	Furthermore, for every $i$, we have $d(v_i,\finv(f(v_i)))\le  K$, and $d(\finv(f(v_i)),\finv(f(v_{i+1})))\le  K$. Therefore, there exists a homotopy
	between $\alpha$ and the image of $f(\alpha)$ by $\finv\circ h$ within the
	$2K$--neighborhood of $\alpha$. It follows that $\alpha$ is 
	null-homologous within $T(K^5,K^{15};\finv(v))$ -- contradiction concluding the proof of the claim.

	\medskip
	
	\noindent
	Let $Y$ be a simplicial complex homeomorphic to an annulus (tube) in $\mathbb E^2$ with
	the inner boundary cycle isomorphic to the boundary cycle $C$ of $D$, and
	admitting a simplicial retraction on $C$. Observe that the boundary cycle
	of $D'$ is also $C$. Let $\ov{D}=D\cup_{C}Y$ be the complex obtained by gluing $D$ and $Y$
	along $C$. Similarly, let $\ov{D'}=D'\cup_{C}Y$. Both, $\ov{D}$ and $\ov{D'}$ are non-singular
	discs, with isomorphic boundaries $C'$ -- the other boundary cycle of $Y$. 
	Consider a triangulated sphere $S:=\ov{D}\cup_{C'} \ov{D'}$ obtained by the identification of the boundaries,
	and the map $\psi \colon  S \to X$ being the union of maps $f$, $g$, and the retraction maps sending copies of $Y$ to their internal cycles $C$.
	By Theorem~\ref{t:sfrc} there exists a simplicial
	extension of $\psi$ to a three-ball without internal vertices. Hence $[\psi]=0$ in 
	$H_2(X_1\cup X_2; \mathbb Z)$. 
	\medskip
	
	On the other hand the $1$--cycle $\alpha$ is null-homotopic inside $B(K^{10},\finv(v))\subseteq
	D$. Hence there exists a simplicial disc $D_1 \subseteq B(K^{10},\finv(v))$ providing the homotopy.
	Similarly, there is a disc $D_2 \subseteq D-B(K^{10},\finv(v))\cup_{C} Y \cup_{C'} \ov{D'}$
	with boundary equal $\alpha$. Observe that $\psi(D_1)\subseteq X_1$, $\psi(D_2)\subseteq X_2$,
	and $\psi(\alpha)\subseteq X_1\cap X_2$. Therefore, in the Mayer-Vietoris sequence for the pair $X_1,X_2$ the boundary map
	\begin{align*}
	H_2(X_1\cup X_2;\mathbb{Z}) \to H_1(X_1\cap X_2; \mathbb Z)
	\end{align*}
	sends $[\psi]$ to the nontrivial element represented by $\alpha$. Hence the contradiction concluding the proof of the lemma.
\end{proof}

\begin{remark}
	In fact, a more general version of Lemma~\ref{t:Morse} could be proved following
	the same lines. Namely, we could require that $f\colon D\to X$ is a disc diagram being a quasi-isometry
	such that $D$ is quasi-isometric to a ball in $\mathbb E_{\triangle}^2$, rather than being the ball itself.
	Since the original statement allows technically much simpler proof, and it is the version that we 
	subsequently use in \cite{Artinsrigidity}, we decided to formulate it this way.
\end{remark}

\section{The complexes for $2$--generated Artin groups}
\label{s:dihedral}
In this section, we focus on $2$--generated Artin groups. We construct metric simplicial complexes for them by modifying their Cayley complexes (see the \textquotedblleft comments on the proof\textquotedblright\ subsection in the Introduction for an intuitive explanation). Later in Section~\ref{s:link} we will show these metric simplicial complexes are metrically systolic, and in Section~\ref{s:general} we will glue them together to form metrically systolic complexes for general two-dimensional Artin groups.

\subsection{Precells in the presentation complex}
\label{subsec:precells}
Let $DA_n$ be the $2$--generator Artin group presented by $\angled{a,b\mid \underbrace{aba\cdots}_{n} =
	\underbrace{bab\cdots}_n}$.

Let $P_n$ be the standard presentation complex for $DA_n$. Namely the $1$--skeleton of $P_n$ is the wedge of two oriented circles, one labeled $a$ and one labeled $b$. Then we attach the boundary of a closed $2$--cell $C$ to the $1$--skeleton with respect to the relator of $DA_n$. Let $C\to P_n$ be the attaching map. Let $\Xa_n$ be the universal cover of $P_n$. Then any lift of the map $C\to P_n$ to $C\to \Xa_n$ is an embedding (cf.\ \cite[Corollary 3.3]{Artinsystolic}). These embedded discs in $\Xa_n$ are called \emph{precells}. Figure~\ref{f:precell} depicts a precell $\Pa$. $\Xa_n$ is a union of copies of $\Pa$'s.
\begin{figure}[h!]
	\centering
	\includegraphics[width=1\textwidth]{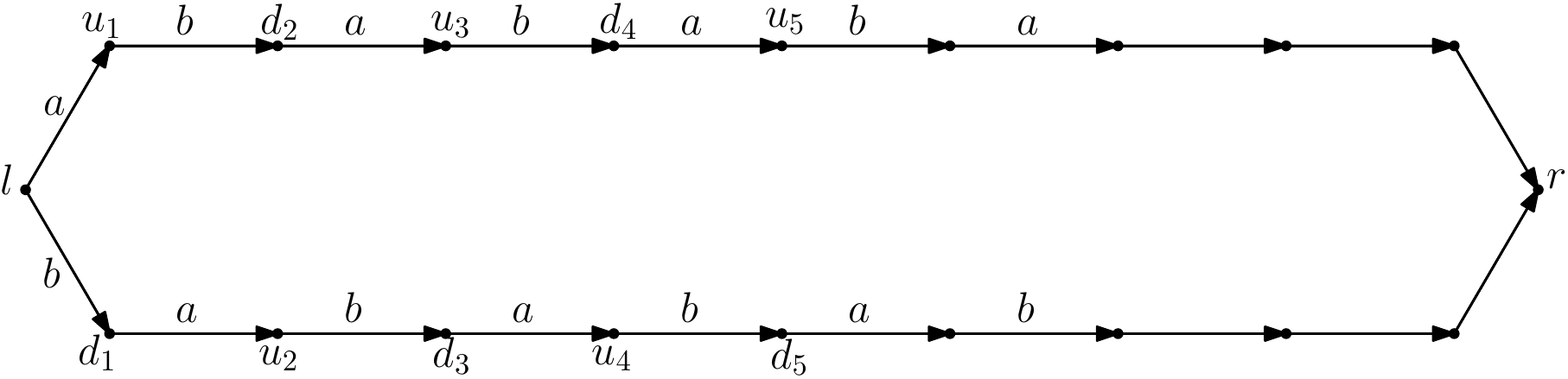}
	\caption{Precell $\Pa$.}
	\label{f:precell}
\end{figure}
We pull back the labeling and orientation of edges in $P_n$ to obtain labeling and orientation of edges in $\Xa_n$.
We label the vertices of $\Pa$ as in Figure \ref{f:precell}. The vertices $\ell$ and $r$ are called the \emph{left tip} and the \emph{right tip} of $\Pa$. The boundary $\partial\Pa$ is made of two paths. The one starting at $\ell$, going along $\underbrace{aba\cdots}_{n}$ (resp.\ $\underbrace{bab\cdots}_{n}$), and ending at $r$ is called the \emph{upper half} (resp.\ \emph{lower half}) of $\partial\Pa$. The orientation of edges inside one half is consistent, thus each half has an orientation. 
We summarize several basic properties of how these precells intersect each other. See \cite[Section 3.1]{Artinsystolic} for proofs of these properties.
\begin{lemma}
	\label{cor:connected intersection}
Let $\Pa_1$ and $\Pa_2$ be two different precells in $\Xa_n$. Then 
\begin{enumerate}
	\item either $\Pa_1\cap\Pa_2=\emptyset$, or $\Pa_1\cap\Pa_2$ is connected;
	\item if $\Pa_1\cap\Pa_2\neq\emptyset$, $\Pa_1\cap\Pa_2$ is properly contained in the upper half or in the lower half of $\Pa_1$ (and $\Pa_2$);
	\item if $\Pa_1\cap\Pa_2$ contains at least one edge, then one end point of $\Pa_1\cap\Pa_2$ is a tip of $\Pa_1$, and another end point of $\Pa_1\cap\Pa_2$ is a tip of $\Pa_2$, moreover, among these two tips, one is a left tip and one is a right tip.
\end{enumerate}
\end{lemma}

\begin{lemma}
	\label{cor:disjoint}
Suppose there are three precells $\Pa_1$, $\Pa_2$ and $\Pa_3$ such that $\Pa_1\cap \Pa_2$ is a nontrivial path $P_1$ in the upper half of $\Pa_2$, and $\Pa_3\cap \Pa_2$ is a nontrivial path $P_3$ in the lower half of $\Pa_2$. Then $\Pa_1\cap \Pa_3$ is either empty or one point.
\end{lemma}

\begin{corollary}
	\label{cor:unique}
	Let $\Pa_1$ and $\Pa_2$ be two different precells in $\Xa_n$. If $\Pa_1\cap\Pa_2$ contains at least one edge, and $\Pa_3\cap\Pa_2=\Pa_1\cap\Pa_2$, then $\Pa_3=\Pa_1$. 
\end{corollary}

\begin{proof}
We apply Lemma~\ref{cor:connected intersection} (3) to $\Pa_3\cap\Pa_2$ and $\Pa_1\cap\Pa_2$ to deduce that either $\Pa_1$ and $\Pa_3$ have the same left tip, or they have the same right tip. Thus $\Pa_1=\Pa_3$.
\end{proof}

\subsection{Subdividing and systolizing the presentation complex}
\label{subsec:subividing and adding new edges}

We subdivide each precell in $\Xa_n$ as in Figure~\ref{f:cell} to obtain a simplicial complex $\Xb_n$. A \emph{cell} of $\Xb_n$ is defined to be a subdivided precell, and we use the symbol $\Pi$ for a cell. 
The original vertices of $\Xa_n$ in $\Xb_n$ are called the \emph{real vertices}, and the new vertices of $\Xb_n$ after subdivision are called \emph{interior vertices}. The interior vertex in a cell $\Pi$ is denoted $o$ as in Figure~\ref{f:cell}. (Here and further we use the convention that the real vertices are drawn as solid points and the interior vertices as circles.)

\begin{figure}[ht!]
	\centering
	\includegraphics[scale=0.8]{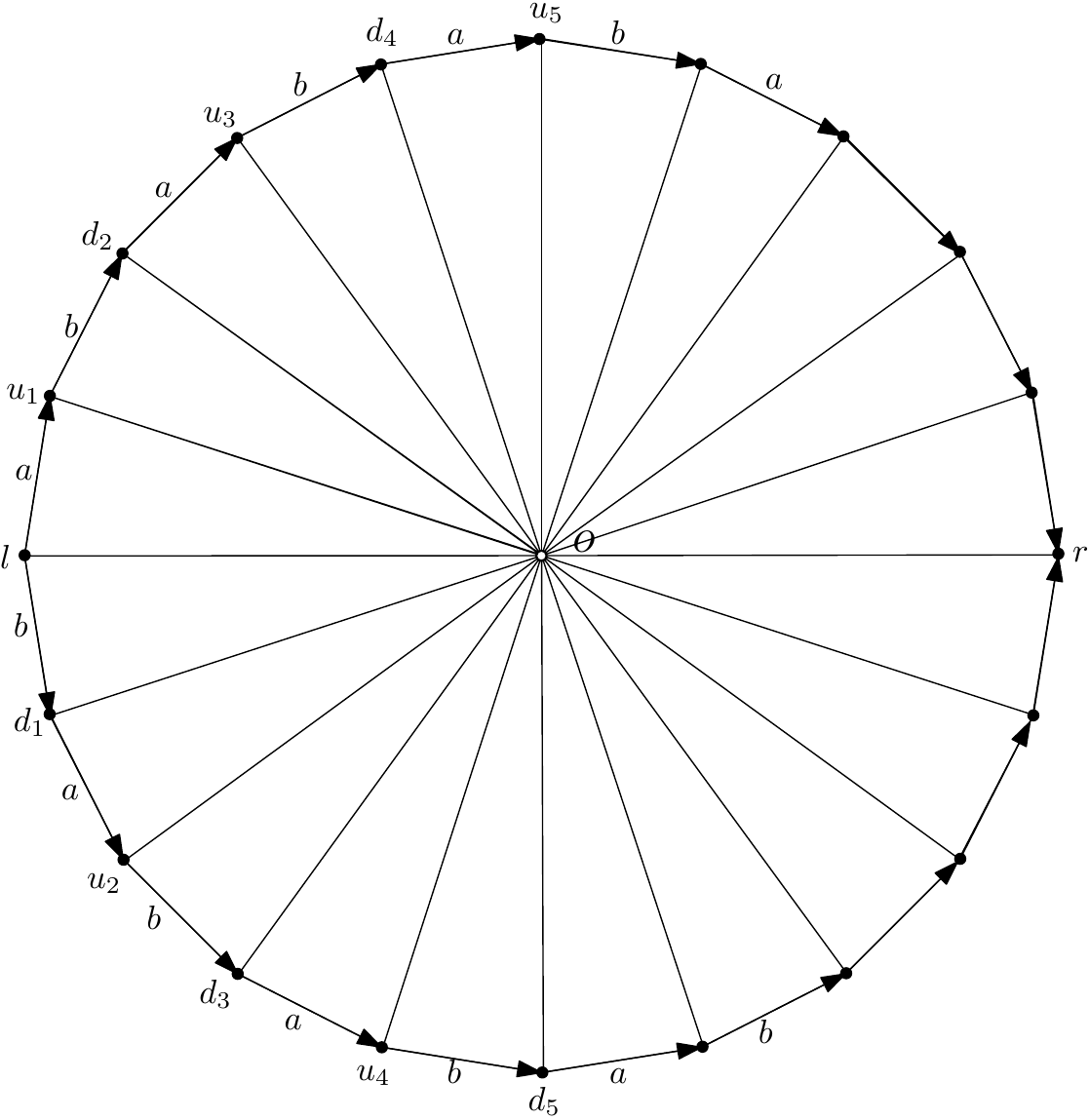}
	\caption{Cell $\Pi$}
	\label{f:cell}
\end{figure}

Let $\Lambda$ be the collection of all unordered pairs of cells of $\Xb_n$ such that their intersection contains at least two edges (these intersections are connected by Lemma~\ref{cor:connected intersection}). For each $(\Pi_1,\Pi_2)\in\Lambda$, we add an edge between the interior vertex of $\Pi_1$ and the interior vertex of $\Pi_2$ (cf.\ Figure~\ref{f:explanation}). Denote the resulting complex by $X'_n$. It is clear that $DA_n$ acts on $X'_n$. Let $X_n$ be the flag completion of $X'_n$. Then $X_n$ is the simplicial complex we will work with.

Now we give an alternative, but more detailed definition of $X'_n$. Pick a base cell $\Pi$ in $\Xb_n$ such that $\ell\in\Pi$ coincides with the identity element of $DA_n$. Let $\Lambda_0$ be the collection of pairs of the form $(\Pi, \ui{i}\Pi)$, $(\Pi, \di{i}\Pi)$ for $i=1,\ldots,n-2$ (here each vertex of $\Pi$ can be identified as an element of $DA_n$, and $\ui{i}\Pi$ means the image of $\Pi$ under the action of $\ui{i}$). Then the following is proved in \cite[Section 3.1]{Artinsystolic}.

\begin{lemma}\
	\label{lem:representative}
\begin{enumerate}
	\item $\Lambda_0\subset\Lambda$.
	\item Different elements in $\Lambda_0$ are in different $DA_n$--orbits.
	\item Every $DA_n$--orbit in $\Lambda$ contains an element from $\Lambda_0$.
\end{enumerate}
\end{lemma}

For each $1\le i\le n-2$, we add an edge between $o\in\Pi$ and $\ui{i}o\in\ui{i}\Pi$, and an edge between $o\in\Pi$ and $\di{i}o\in\di{i}\Pi$. Then we use the action of $DA_n$ to add more edges in the equivariant way. The resulting complex is exactly $X'_n$, by Lemma~\ref{lem:representative}.

\begin{definition}
	\label{def:length}
We assign lengths to edges of $X_n$. Edges between a real vertex and an interior vertex have length 1. Edges between two real vertices have length equal to the distance between two adjacent vertices in a regular $(2n)$--gon with radius 1. 

Now we assign lengths to edges between two interior vertices. First define a function $\phi:[0,\pi)\to \mathbb R$ as follows. Let $\Delta(ABC)$ be a Euclidean isosceles triangle with length of $AB$ and $AC$ equal to 1, and $\angle_A(B,C)=\alpha$. Then $\phi(\alpha)$ is defined to be the length of $BC$. For $1\le i\le n-2$, let $e_i$ be the edge between $o$ and $\ui{i}o$ (or $o$ and $\di{i}o$).  Then the length of $e_i$ is defined to be $\phi(\frac{i}{2n}2\pi)$. Now we use the $DA_n$ action to define the length of edges between interior vertices in an equivariant way.
\end{definition}

Note that $\Pi\cap \ui{i}\Pi$ and $\Pi\cap \di{i}\Pi$ have $n-i$ edges. Thus we have the following observation by using the $DA_n$--action and Lemma~\ref{lem:representative}.

\begin{lemma}
	\label{lem:overlap and length}
Suppose $\Pi_1\cap \Pi_2$ has $m$ edges for $m\ge 2$. Let $o_i\in\Pi_i$ be the interior vertex for $i=1,2$. Then there is an edge between $o_i$ and $o_j$ in $X_n$ whose length is $\phi(\frac{n-m}{2n}2\pi)$.
\end{lemma}

\begin{lemma}
	\label{lem:strict triangle inequality}
The lengths of the three sides of each triangle in $X^{(1)}_n$ satisfy the strict triangle inequality. Thus each $2$--simplex of $X_n$ can be metrized as a non-degenerate Euclidean triangle whose three sides have length equal to the assigned length of the corresponding edges.
\end{lemma}

\begin{proof}
We only prove the case when this triangle is made of three interior vertices $\{o_i\in\Pi_i\}_{i=1}^3$. The other cases are already clear from the construction. By Lemma~\ref{cor:disjoint}, $\Pi_1\cap\Pi_2$ and $\Pi_1\cap \Pi_3$ are contained in the same half (say upper half) of $\Pi_1$, otherwise $\Pi_2\cap \Pi_3$ is at most one vertex, which contradicts that $o_2$ and $o_3$ are joined by an edge. We assume without loss of generality that $\Pi_1$ is the base cell $\Pi$. By Lemma~\ref{cor:connected intersection} (3), each of $\Pi_2$ and $\Pi_3$ contains exactly one tip of $\Pi_1$. We first consider the case when $\Pi_2$ contains the left tip of $\Pi_1$ and $\Pi_3$ contains the right tip of $\Pi$. Suppose $\Pi_2\cap \Pi_1$ (resp.\ $\Pi_3\cap\Pi_1$) contains $m_2$ (resp.\ $m_3$) edges. Then by Lemma~\ref{cor:connected intersection} (3), $\Pi_2\cap \Pi_3$ contains $m_2+m_3-n$ edges. By Lemma~\ref{lem:overlap and length}, length$(\overline{o_1o_2})=\phi(\frac{n-m_2}{2n} 2\pi)$, length$(\overline{o_1o_3})=\phi(\frac{n-m_3}{2n} 2\pi)$, and length$(\overline{o_2o_3})=\phi(\frac{n-(m_2+m_3-n)}{2n} 2\pi)$. Note that $\pi>\frac{n-(m_2+m_3-n)}{2n} 2\pi=\frac{n-m_2}{2n} 2\pi+\frac{n-m_3}{2n} 2\pi$, thus we can place $o_2,o_1,o_3$ consecutively in the unit circle such that they span a Euclidean triangle with side lengths as required. Next we consider the case that both $\Pi_2$ and $\Pi_3$ contains the left tip of $\Pi_1$. We assume without loss of generality that $\Pi_1\cap\Pi_2\subsetneq \Pi_1\cap\Pi_3$. Then, by Corollary~\ref{cor:connected intersection} (3), the left tip of $\Pi_3$ is contained in $\Pi_2\cap\Pi_3$. Thus we can repeat the argument in the previous case with $\Pi_1$ replaced by $\Pi_3$. The case when both $\Pi_2$ and $\Pi_3$ contain the right tip of $\Pi_1$ can be handled similarly.
\end{proof}

From now on, we think of each $2$--simplex of $X_n$ as a Euclidean triangle with the required side lengths. If three vertices $x_1$, $x_2$ and $x_3$ span a $2$--simplex in $X_n$, then we use $\angle_{x_1}(x_2,x_3)$ to denote the angle at $x_1$ of the associated Euclidean triangle.

\section{The link of $X_n$}
\label{s:link}
In this section we study links of vertices in the complex $X_n$ defined in the previous section.

Choose a vertex $v\in X_n$, let $\Lambda_v$ be the link $lk(v,X^{(1)}_n)$ of $v$ in $X_n$, i.e.\ $\Lambda_v$ is the full subgraph of $X^{(1)}_n$ spanned by vertices which are adjacent to $v$. For an edge $\overline{v_1v_2}\subset\Lambda_v$, we define the \emph{angular length} of this edge to be $\angle_v(v_1,v_2)$. This makes $\Lambda_v$ a metric graph. We define angular metric on $\Lambda_v$ in the same way as in Subsection~\ref{s:msdef} and use the notation from over there. 

The main result of the section is the following proposition.

\begin{prop} 
	\label{prop:real link}
	Let $v$ be a vertex of $X_n$.
	\begin{enumerate}
		\item The angular lengths of the three sides of each triangle in $\Lambda_v$ satisfy the triangle inequality.
		\item Let $\sigma$ be a simple cycle in $\Lambda_v$ which is $2$--full. Then $\length_\angle(\sigma)\ge 2\pi$.
	\end{enumerate}
\end{prop}
%

We caution the reader that each edge in $\Lambda_v$ has an angular length, and has a length as defined in the previous section. Here we mostly work with angular length, but will switch to length occasionally. In this section we study the structure of $\Lambda_v$ with respect to the angular metric.

The proof of Proposition~\ref{prop:real link} is divided into two cases: the case of a real vertex $v$ is treated in Subsection~\ref{s:real} and the case of an interior vertex $v$ is treated in Subsection~\ref{s:interior}. In each case we first describe precisely the combinatorial and metric structure of the link and then we study in details angular lengths of simple cycles in the link.  

\subsection{Link of a real vertex}
\label{s:real}
The main purpose of this subsection is to prove Proposition~\ref{prop:real link} for a real vertex $v$.

Since the links of any two real vertices are isomorphic as metric graphs with the angular metric, we can assume without loss of generality that $v$ is the vertex $l$ in the base cell $\Pi$ (cf. Figure~\ref{f:cell}). 

In the following proof, we will assume $u_0=d_0=\ell$ and $u_n=d_n=r$. Recall that each edge of $X_n$ which belongs to $X^\ast_n$ has an orientation and is labeled by one of the generators $a$ and $b$. We will first establish a sequence of lemmas towards the proof of Proposition~\ref{prop:real link}.

The vertices of $\Lambda_v$ can be divided into two classes.
\begin{enumerate}
	\item Real vertices $a^i,a^o,b^i$ and $b^o$, where $a^i$ and $a^o$ are the vertices in $\Lambda_v$ which correspond to the incoming and outgoing $a$--edge containing $v$ ($b^i$ and $b^o$ are defined similarly).
	\item Interior vertices. There is a 1-1 correspondence between such vertices and cells in $X_n$ that contain $\ell$. Thus the interior vertices of $\Lambda_v$ are of form $w^{-1}o$ where $w$ is a vertex of $\partial\Pi$ (recall that we have identified vertices of $X^{\ast}_n$ with group elements of $DA_n$, and $\ell$ is identified with the identity element of $DA_n$, so $w^{-1}o$ means the image of $o$ under the action of $w^{-1}$). More precisely, interior vertices of $\Lambda_v$ are $\{\ell^{-1} o,r^{-1}o,d^{-1}_1o,d^{-1}_2o,\ldots,d^{-1}_{n-1}o,u^{-1}_1o,u^{-1}_2o,\ldots,u^{-1}_{n-1}o\}$.
\end{enumerate}

The edges of $\Lambda_v$ can be divided into two classes.
\begin{enumerate}
	\item Edges between a real vertex and an interior vertex. These are exactly the edges of $\Lambda_v$ which are in $\Xb_n$, and they are called edges of \emph{type I}.
	\item Edges between two interior vertices. These are exactly the edges of $\Lambda_v$ which are not in $\Xb_n$, and they are called edges of \emph{type II}.
\end{enumerate}
Note that there do not exist edges of $\Lambda_v$ which are between two real vertices. 

Now we characterize all edges of type I. See Figure~\ref{f:real} below for a picture of $\Lambda_v$ with only edges of type I shown.
\begin{lemma}
	\label{lem:edge type I}
	\begin{enumerate}
		\item The collection of vertices in $\Lambda_v$ which are connected to $b^i$ (resp.\ $a^i$) by an edge of type $I$ is exactly $\{d^{-1}_1o,d^{-1}_2o,\ldots,d^{-1}_no\}$ (resp.\ $\{u^{-1}_1o,u^{-1}_2o,\ldots,u^{-1}_no\}$).
		\item The collection of vertices in $\Lambda_v$ which are connected to $a^o$ (resp.\ $b^o$) by an edge of type $I$ is exactly $\{d^{-1}_0o,d^{-1}_1o,\ldots,d^{-1}_{n-1}o\}$ (resp.\ $\{u^{-1}_0o,u^{-1}_1o,\ldots,u^{-1}_{n-1}o\}$).
		\item Each edge of type I has angular length equal to $\frac{n-1}{4n}2\pi$.
	\end{enumerate}
\end{lemma}

\begin{proof}	
	If a vertex in $\Lambda_v$ is adjacent to $b^i$, then this vertex must be an interior vertex, hence is of form $w^{-1}o$ for a vertex $w\in\partial\Pi$. Note that if there is a vertex $w'\in\partial\Pi$ such that there is a $b$--edge pointing from $w'$ to $w$, then by applying the action of $w^{-1}$ to the triangle $\Delta(w'wo)$, we know $w^{-1}o$ and $b^i$ are adjacent. We can reverse this argument to show that $w^{-1}o$ and $b^i$ are adjacent, then there is a $b$--edge in $\partial\Pi$ terminating at $w$. It follows that $b^i$ is connected to $w^{-1}o$ if and only if $u=d^{-1}_i$ for $1\le i\le n$. Thus the part of (1) concerning $b^i$ follows. We can analyze vertices to $b^o,a^i$ and $a^o$ in a similar way. Thus (1) and (2) follow. Note that the angular length of each edge of type I is equal to half of the interior angle of a regular $2n$--gon. Thus (3) follows.
\end{proof}

\begin{figure}[ht!]
	\centering
	\includegraphics[width=1\textwidth]{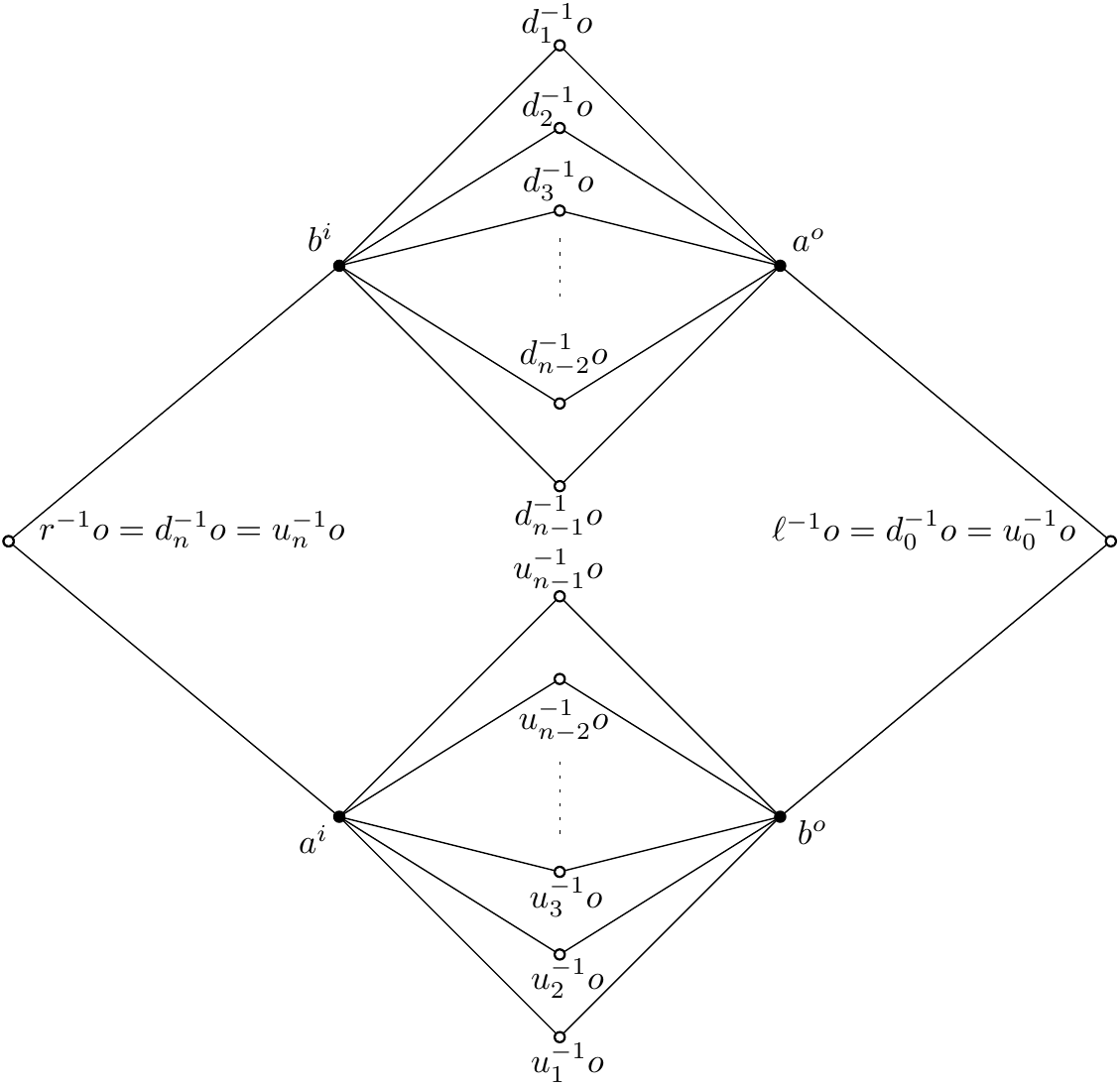}
	\caption{Edges of type I in the link of a real vertex}
	\label{f:real}
\end{figure}

\begin{lemma}\
	\label{lem:edge type II}
	\begin{enumerate}
		\item There is an edge of type II between $d^{-1}_io$ and $d^{-1}_jo$ if and only if $1\le |j-i|\le n-2$.
		\item There is an edge of type II between $u^{-1}_io$ and $u^{-1}_jo$ if and only if $1\le |j-i|\le n-2$.
		\item If $1\le i\le n-1$ and $1\le j\le n-1$, then there is no edge between $d^{-1}_io$ and $u^{-1}_jo$.
		\item Suppose $0\le i<j\le n$ and $j-i\le n-2$. Then the edge between $d^{-1}_io$ and $d^{-1}_jo$ has angular length $=\frac{j-i}{2n}2\pi$.
		\item Suppose $0\le i<j\le n$ and $j-i\le n-2$. Then the edge between $u^{-1}_io$ and $u^{-1}_jo$ has angular length $=\frac{j-i}{2n}2\pi$.
	\end{enumerate}
\end{lemma}

\begin{proof}
	First we claim the number of edges in $d^{-1}_i\Pi\cap d^{-1}_j\Pi$ equals to $n-(|j-i|)$. We assume without loss of generality that $i<j$. Then the number of edges in $d^{-1}_i\Pi\cap d^{-1}_j\Pi$ equals the number of edges in $\Pi\cap d_id^{-1}_j\Pi$. By direct computation, we know $d_id^{-1}_j=u^{-1}_{j-i}$ or $d^{-1}_{j-i}$. Moreover, the number of edges in $\Pi\cap d^{-1}_k\Pi$ (or $\Pi\cap u^{-1}_k\Pi$) equals  $n-k$ for any $1\le k\le n-1$. Thus the claim follows.
	
	There is an edge of type II between $d^{-1}_io$ and $d^{-1}_jo$ if and only if $d^{-1}_i\Pi\cap d^{-1}_j\Pi$ has at least two edges, thus (1) follows from the claim. (4) follows the claim and Lemma~\ref{lem:overlap and length}. (2) and (5) can be proved in a similar way. To see (3), note that $d^{-1}_i\Pi\cap\Pi$ (resp.\ $u^{-1}_i\Pi\cap\Pi$) is contained in the upper half (resp.\ lower half) of $\Pi$. Thus (3) follows from Lemma~\ref{cor:disjoint}.
\end{proof}

\begin{corollary}
	\label{cor:triangle inequality1}
	\begin{enumerate}
		\item The angular lengths of the three sides of any triangle in $\Lambda_v$ satisfy the triangle inequality. 
		\item Let $\Delta$ be a $3$--simplex in $X_n$ which contains a real vertex. Then there exists a (possibly degenerate) $3$--simplex $\Delta'$ in the Euclidean $3$--space such that there is a simplicial isomorphism $\Delta\to\Delta'$ which preserves the lengths of edges.
	\end{enumerate}
\end{corollary}

\begin{proof}
	Let $\Delta$ be a triangle in $\Lambda_v$. Since no two real vertices in $\Lambda_v$ are adjacent, $\Delta$ either has two interior vertices, or three interior vertices. In the former case, since the angular length of any edge of type II is at most $\frac{n-2}{2n}2\pi$ (Lemma~\ref{lem:edge type II}), it is less than the summation of the angular length of two edges of type I (Lemma~\ref{lem:edge type I}), we consequently deduce that the triangle inequality holds. Moreover, (2) holds by triangle inequality and that the summation of the angular length of two edges of type I in $\Delta$ is $<\pi$. In the latter case, by Lemma~\ref{lem:edge type II} (3), the three vertices of $\Delta$ are either of form $\di{i}o,\di{j}o,\di{k}o$, or of form $\ui{i}o,\ui{j}o,\ui{k}o$. By Lemma~\ref{lem:edge type II} (4), $\angle_v(\di{i}o,\di{j}o)+\angle_v(\di{j}o,\di{k}o)=\angle_v(\di{i}o,\di{k}o)$ when $i<j<k$. A similar equality holds with $d$ replaced by $u$. Thus (1) and (2) follow.
\end{proof}

We record a simple graph theoretic observation for later use.
\begin{definition}
	\label{def:tree}
	A simple graph $\Gamma$ is a \emph{tree of cliques} if there are complete subgraphs $\{\Delta_i\}_{i=1}^k$ such that 
	\begin{enumerate}
		\item $\Gamma=\cup_{i=1}^k\Delta_i$;
		\item for each $1<m\le k$, $(\cup_{i=1}^{m-1}\Delta_i)\cap\Delta_m$ is a complete subgraph.
	\end{enumerate}
\end{definition}

\begin{lemma}
	\label{lem:cycle in tree}
	Let $\Gamma$ be a tree of cliques. Then the following hold.
	\begin{enumerate}
		\item Any simple $n$--cycle for $n\ge 4$ in $\Gamma$ is not $2$--full.
		\item If $\Gamma$ is a metric graph such that the three sides of each of its triangle satisfy the triangle inequality then, for any edge $\overline{w_1w_2}\subset\Gamma$, the length of $\overline{w_1w_2}$ is bounded above by the length of any edge path from $w_1$ to $w_2$.
	\end{enumerate}
\end{lemma}

\begin{proof}
	For (1), we induct on the number $k$ in Definition~\ref{def:tree}. Let $\omega\subset\Gamma$ be a simple $n$--cycle. If $\omega\subset\cup_{i=1}^{k-1}\Delta_i$, then $\omega$ is not $2$--full by induction. Now we assume $\omega\nsubseteq\cup_{i=1}^{k-1}\Delta_i$. Then there must be an edge $e\subset \omega$ such that $e\nsubseteq\cup_{i=1}^{k-1}\Delta_i$. Let $s,t$ be two vertices of $e$. By Definition~\ref{def:tree} (1), $e\subset \Delta_k$. Hence $\{s,t\}\subset \Delta_k$. If $\{s,t\}\subset \cup_{i=1}^{k-1}\Delta_i$, then by Definition~\ref{def:tree} (2) and the assumption that $\Gamma$ is simple, we know $e\subset (\cup_{i=1}^{k-1}\Delta_i)\cap\Delta_k$, which is a contradiction. So at least one of $\{s,t\}$ is not contained in $\cup_{i=1}^{k-1}\Delta_i$. Now we assume $s\in \Delta_k\setminus (\cup_{i=1}^{k-1}\Delta_i)$. Let $t_1$ and $t_2$ be two vertices in $\omega$ that are adjacent to $s$. Since $n\ge 4$, $t_1$ and $t_2$ have combinatorial distance $\ge 2$ in $\omega$. By Definition~\ref{def:tree} (1), the edge $\overline{t_1s}$ is contained in one of the $\Delta_i$. Thus we must have $\overline{t_1s}\subset \Delta_k$. In particular $t_1\in\Delta_k$. Similarly, $t_2\in\Delta_k$. Thus there is an edge between $t_1$ and $t_2$, and $\omega$ is not $2$--full.
	
	For (2), we can assume without loss of generality that $\overline{w_1w_2}$ together with another given edge path from $w_1$ to $w_2$ form a simple cycle. Thus it suffices to show that for any simple cycle $\omega\subset \Gamma$, the length of an edge $e\in\omega$ is bounded above by the summation of the lengths of other edges in $\omega$. Let $n$ be the number of edges in $\omega$. We induct on $n$. The case $n=3$ follows from the assumption. The case $n\ge 4$ follows from the induction assumption and from the fact that $\omega$ is not $2$--full.
\end{proof}

Let $\Lambda^+_v$ be the full subgraph of $\Lambda_v$ spanned by $\{b^i,a^o,\di{0}o,\di{1}o,\ldots,\di{n}o\}$. Let $\Lambda^-_v$ be the full subgraph of $\Lambda_v$ spanned by $\{b^o,a^i,\ui{0}o,\ui{1}o,\ldots,\ui{n}o\}$.

\begin{lemma}
	\label{lem:half}
	Each of $\Lambda^+_v$ and $\Lambda^-_v$ is a tree of cliques.
\end{lemma}

\begin{proof}
	We define the following sets of vertices of $\Lambda^+_v$.
	\begin{enumerate}
		\item $V_1=\{b^i,\di{n}o,\di{n-1}o,\ldots,\di{2}o\}$;
		\item $V_2=\{b^i,\di{n-1}o,\di{n-2}o,\ldots,\di{1}o\}$;
		\item $V_3=\{a^o,\di{n-1}o,\di{n-2}o,\ldots,\di{1}o\}$;
		\item $V_4=\{a^o,\di{n-2}o,\di{n-3}o,\ldots,\di{0}o\}$.
	\end{enumerate}
	By Lemma~\ref{lem:edge type I} and Lemma~\ref{lem:edge type II}, each $V_i$ spans a complete subgraph, which we denote by $\Delta_i$. Moreover, $\Lambda^+_v=\Delta_1\cup\Delta_2\cup\Delta_3\cup\Delta_4$. Definition~\ref{def:tree} (2) can be verified directly. Thus $\Lambda^+_v$ is a tree of cliques. Similarly, $\Lambda^-_v$ is a tree of cliques.
\end{proof}

\begin{lemma}
	\label{lem:$2$--full cycle contains tips}
	Let $\sigma\subset\Lambda_v$ be a simple cycle such that $\sigma\nsubseteq\Lambda^+_v$ and $\sigma\nsubseteq\Lambda^-_v$. Then $\ell^{-1}o\in\sigma$ and $r^{-1}o\in\sigma$. Consequently, if $\sigma$ is $2$--full simple $n$--cycle in $\Lambda_v$ for $n\ge 4$, then $\ell^{-1}o\in\sigma$ and $r^{-1}o\in\sigma$. 
\end{lemma}

\begin{proof}
	It follows from Lemma~\ref{lem:edge type I} and Lemma~\ref{lem:edge type II} (3) that there are no edges between a vertex in $\Lambda^+_v\setminus \{\ell^{-1}o,r^{-1}o\}$ and a vertex in $\Lambda^-_v\setminus \{\ell^{-1}o,r^{-1}o\}$. Thus vertices of $\Lambda^+_v\setminus \{\ell^{-1}o,r^{-1}o\}$ and vertices of $\Lambda^-_v\setminus \{\ell^{-1}o,r^{-1}o\}$ are in two different connected components of $\Lambda_v\setminus \{\ell^{-1}o,r^{-1}o\}$. Since $\sigma$ is a simple cycle, it follows that at least one of the following three situations happens: (1) $\sigma\subset \Lambda^+_v$; (2) $\sigma\subset \Lambda^-_v$; (3) $r^{-1}o\in\sigma$ and $\ell^{-1}o\in\sigma$. Thus the first statement follows. Lemma~\ref{lem:half} and Lemma~\ref{lem:cycle in tree} imply that (1) and (2) are not possible, thus the second statement follows.
\end{proof}

\begin{lemma}
	\label{lem:pi lower bound}
	Any edge path in $\Lambda_v$ from $r^{-1}o$ to $\ell^{-1}o$ has angular length $\ge \pi$. 
\end{lemma}

\begin{proof}
	Let $\omega$ be an edge path from $r^{-1}o$ to $\ell^{-1}o$. Since vertices of $\Lambda^+_v\setminus\{\ell^{-1}o,r^{-1}o\}$ and vertices of $\Lambda^-_v\setminus\{\ell^{-1}o,r^{-1}o\}$ are in different components of $\Lambda_v\setminus\{\ell^{-1}o,r^{-1}o\}$, there is a sub-path $\omega'\subset \omega$ traveling from $r^{-1}o$ to $\ell^{-1}o$ such that $\omega'\subset \Lambda^+_v$ or $\omega'\subset \Lambda^-_v$. So it suffices to show any edge path $\omega$ in $\Lambda^+_v$ or $\Lambda^-_v$ from $r^{-1}o$ to $\ell^{-1}o$ has angular length $\ge \pi$. We only prove the case $\omega\subset\Lambda^+_v$ since the other case is similar. Note that $\omega$ has to pass through at least one vertex in $\{\di{i}o\}_{i=1}^{n-1}$, so we can divide into the following four cases.
	
	\emph{Case 1:} If there exists $1< k< n-1$ such that $\di{k}o\in\omega$, then Lemma~\ref{lem:cycle in tree} (2), Lemma~\ref{lem:half} and Lemma~\ref{lem:edge type II} imply that $\length_\angle(\omega)\ge\angle_v(\di{n}o,\di{k}o)+\angle_v(\di{k}o,\di{0}o)=\frac{n-k}{2n}2\pi+\frac{k}{2n}2\pi=\pi$. 
	
	\emph{Case 2:} If both $\di{1}o$ and $\di{n-1}o$ are in $\omega$, then Lemma~\ref{lem:cycle in tree} (2), Lemma~\ref{lem:half} and Lemma~\ref{lem:edge type II} imply that $\length_\angle(\omega)\ge\angle_v(\di{n}o,\di{n-1}o)+\angle_v(\di{n-1}o,\di{1}o)+\angle_v(\di{1}o,\di{0}o)=\frac{1}{2n}2\pi+\frac{n-2}{2n}2\pi+\frac{1}{2n}2\pi=\pi$. 
	
	\emph{Case 3:} Suppose among $\{\di{i}o\}_{i=1}^{n-1}$, only $\di{1}o$ is inside $\omega$. Then we must have $b^i\in \omega$ (since there has to be a vertex in $\omega$ which is adjacent to $r^{-1}o$). Thus $\length_\angle(\omega)\ge\angle_v(\di{n}o,b^i)+\angle_v(b^i,\di{1}o)+\angle_v(\di{1}o,\di{0}o)=\frac{n-1}{4n}2\pi+\frac{n-1}{4n}2\pi+\frac{1}{2n}2\pi=\pi$. 
	
	\emph{Case 4:} Suppose among $\{\di{i}o\}_{i=1}^{n-1}$, only $\di{n-1}o$ is inside $\omega$. This can be dealt in the same way as the previous case.
\end{proof}

\begin{proof}[Proof of Proposition~\ref{prop:real link} (for real vertices)]
	Proposition~\ref{prop:real link} (1) follows from Corollary~\ref{cor:triangle inequality1} and (2) follows from Lemma~\ref{lem:$2$--full cycle contains tips} and Lemma~\ref{lem:pi lower bound}.
\end{proof}

The following lemma will be used in Section~\ref{s:general}.

\begin{lemma}\
	\label{lem:distance between real vertices}
	\begin{enumerate}
		\item $d_\angle(a^i,b^i)=d_\angle(a^i,b^o)=d_\angle(a^o,b^i)=d_\angle(a^o,b^o)=\frac{n-1}{2n} 2\pi$.
		\item $d_\angle(a^i,a^o)=d_\angle(b^i,b^o)=\pi$.
	\end{enumerate}
\end{lemma}
Recall that $d_\angle$ denotes the angular metric on $\Lambda_v$.
\begin{proof}
	Note that all edges of type II are between two interior vertices, and there are no edges between real vertices. Thus to travel from one real vertex to another real vertex in $\Lambda_v$, one has to go through at least two edges of type I. Then (1) follows from Lemma~\ref{lem:edge type I} (3). Now we prove (2). Still, traveling from $a^i$ to $a^o$ has to go through at least two edges of type I. However, one readily verifies that only two edges of type I do not bring one from $a^i$ to $a^o$. So we need at least one other edge. By Lemma~\ref{lem:edge type I} and Lemma~\ref{lem:edge type II}, an edge in $\Lambda_v$ has angular length at least $\frac{1}{2n}2\pi$. Thus $d_\angle(a^i,a^o)\ge \pi$. On the other hand, the distance $\pi$ can be realized by $a^o\to \di{n-1}o\to r^{-1}o\to a^i$. Thus $d_\angle(a^i,a^o)=\pi$. Similarly, we obtain $d_\angle(b^i,b^o)=\pi$.
\end{proof}

It is natural to ask when an edge path in $\Lambda_v$ from $r^{-1}o$ to $\ell^{-1}o$ has angular length exactly $=\pi$. We record the following simple observation about such edge paths. The following will be crucial for applications in \cite{Artinsrigidity}.

\begin{lemma}
	\label{lem:exactly pi}
	Suppose $v$ is real and $\omega$ is an edge path in $\Lambda_v$ from $r^{-1}o$ to $\ell^{-1}o$ of angular length $\pi$. Then either $\omega\subset\Lambda^+_v$ or $\omega\subset\Lambda^-_v$. If $\omega\subset\Lambda^+_v$, then the following are the only possibilities for $\omega$:
	\begin{enumerate}
		\item $\omega=r^{-1}o\to b^i\to d^{-1}_{1}o\to d^{-1}_{0}o$;
		\item $\omega=d^{-1}_{n}o\to d^{-1}_{n-1}o\to a^o\to \ell^{-1}o$;
		\item $\omega=d^{-1}_{i_1}o\to d^{-1}_{i_2}o\to\cdots\to d^{-1}_{i_k}o$, where $n=i_1>i_2>\cdots>i_k=0$.
	\end{enumerate}
	A similar statement holds for $\omega\subset\Lambda^-_v$.
\end{lemma}

\begin{proof}
	Note that $\omega$ is embedded, otherwise we can pass to a shorter path from $r^{-1}o$ to $\ell^{-1}o$, which contradicts	Lemma~\ref{lem:pi lower bound}. The statement $\omega\subset\Lambda^+_v$ or $\omega\subset\Lambda^-_v$ follows from the fact that there are no edges between a vertex in $\Lambda^+_v\setminus \{r^{-1}o,\ell^{-1}o\}$ and a vertex in $\Lambda^-_v\setminus \{r^{-1}o,\ell^{-1}o\}$. Now we assume $\omega\subset\Lambda^+_v$.
	
	If $\omega$ does not contain any real vertices, then we are in case (3), by Lemma \ref{lem:edge type II} (4). If $\omega$ contains a real vertex, then it contains at least two edges of type I. Note that the angular length of $\omega$ with two edges of type I removed is $\pi-\frac{n-1}{2n}2\pi=\frac{1}{2n}2\pi$, which equals to the smallest possible angular length of edges in $\Lambda_v$. Thus we are in cases (1) or (2).
\end{proof}

\subsection{Link of an interior vertex}
\label{s:interior}
In this subsection we prove Proposition~\ref{prop:real link} for an interior vertex $v$.

We assume without loss of generality that $v$ is the interior vertex $o$ of the base cell $\Pi$. Moreover, we assume $v\in X_n$ for $n\ge 3$, since the $n=2$ case is clear. Vertices of $\Lambda_v$ can be divided into the following two classes.
\begin{enumerate}
	\item Real vertices. These are the vertices in $\partial\Pi$.
	\item Interior vertices. They are the interior vertices of some cell $\Pi'$ such that $\Pi'\cap\Pi$ contains at least two edges.
\end{enumerate}

For the convenience of the proof, we name the vertices in $\partial\Pi$ differently in this subsection. The vertices in the upper half (resp.\ lower half) of $\partial\Pi$ are called $v_0,v_1,\ldots,v_n$ (resp.\ $v'_0,v'_1,\ldots,v'_n$) from left to right. Note that $v_0=v'_0$ and $v_n=v'_n$.

Let $\mathcal{P}$ be the collection of subcomplexes of $\partial\Pi$ such that
\begin{enumerate}
	\item they are homeomorphic to the unit interval $[0,1]$;
	\item each of them has $m$ edges where $2\le m\le n-1$;
	\item each of them is contained in a half of $\partial\Pi$, and has nontrivial intersection with $\{
	\ell,r\}\subset \partial\Pi$.
\end{enumerate}
By Lemma~\ref{cor:connected intersection} (3), for each interior vertex of $\Lambda_v$, the intersection of the cell containing this interior vertex and $\Pi$ is an element in $\mathcal{P}$. This actually induces a one to one correspondence between interior vertices of $\Lambda_v$ and elements of $\mathcal{P}$ by Corollary~\ref{cor:unique}. Thus we can name the interior vertices of $\Lambda_v$ as follows. If the intersection of the cell which contains this interior vertex and $\Pi$ is a path in the upper half (resp.\ lower half) of $\partial\Pi$ that starts at $\ell$ and has $i$ edges, then we denote this interior vertex by $L_i$ (resp.\ $L'_i$). If the intersection of the cell which contains this interior vertex and $\Pi$ is a path in the upper half (resp.\ lower half) of $\partial\Pi$ that ends at $r$ and has $i$ edges, then we denote this interior vertex by $R_i$ (resp.\ $R'_i$). Note that $i$ is ranging from $2$ to $n-1$; see Figure~\ref{f:interior}. Let $\Pi_{L_i}$ be the cell that contains $L_i$. We define $\Pi_{L'_i},\Pi_{R_i}$ and $\Pi_{R'_i}$ similarly. 
\begin{figure}[ht!]
	\centering
	\includegraphics[width=1\textwidth]{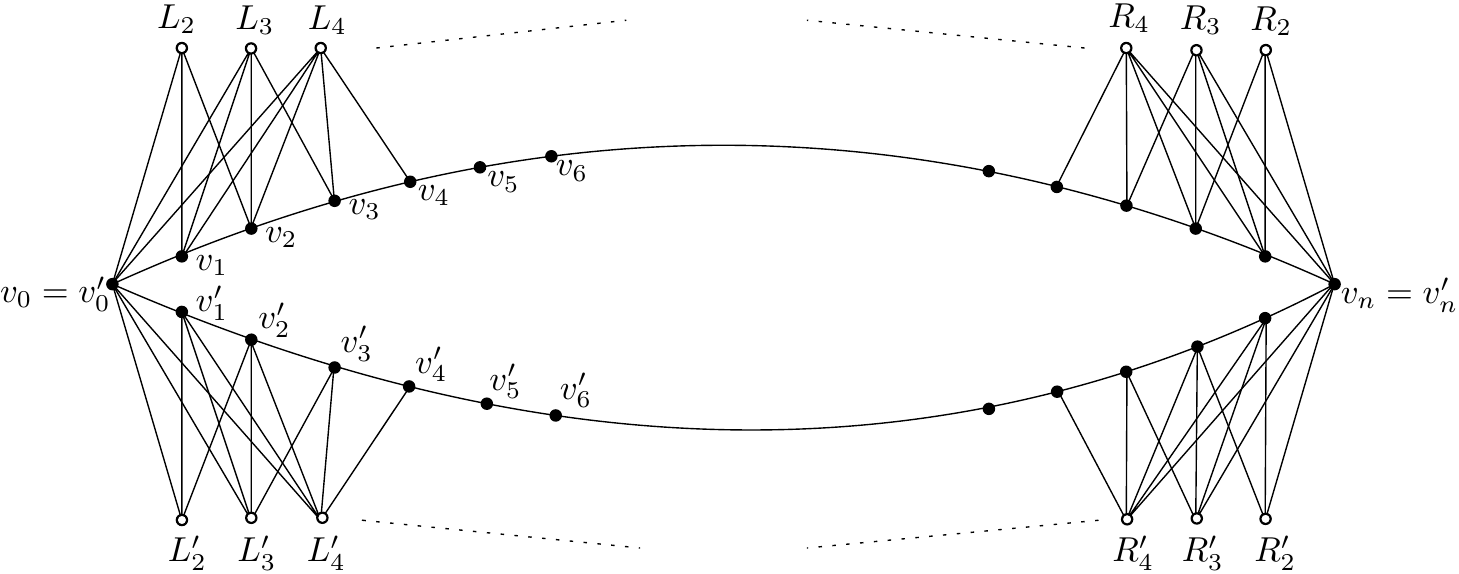}
	\caption{The link of an interior vertex}
	\label{f:interior}
\end{figure}

Now we characterize edges in $\Lambda_v$. They are divided into three classes.
\begin{enumerate}
	\item \emph{Edges of type I}. They are edges between real vertices of $\Lambda_v$. Hence they are exactly edges in $\partial\Pi$. Each of them has angular length $=\frac{1}{2n}2\pi$.
	\item \emph{Edges of type II}. They are edges between a real vertex and an interior vertex, and they are characterized by Lemma~\ref{lem:real interior} below.
	\item \emph{Edges of type III}. They are edges between interior vertices of $\Lambda_v$, and they are characterized by Lemma~\ref{lem:two interiors} below.
\end{enumerate}
We refer to Figure~\ref{f:interior} for a picture of $\Lambda_v$. Edges of type I and some edges of type II are drawn, but edges of type III are not drawn in the picture.

\begin{lemma}\
	\label{lem:real interior}
	\begin{enumerate}
		\item The collection of vertices in $\partial\Pi$ that are adjacent to $L_i$ (resp.\ $L'_i$) is $\{v_0,v_1,\ldots,v_i\}$ (resp.\ $\{v'_0,v'_1,\ldots,v'_i\}$). 
		\item The collection of vertices in $\partial\Pi$ that are adjacent to $R_i$ (resp.\ $R'_i$) is $\{v_n,v_{n-1},\ldots,v_{n-i}\}$ (resp.\ $\{v'_n,v'_{n-1},\ldots,v'_{n-i}\}$). 
		\item The angular length of any edge between $L_i$ and a real vertex of $\Lambda_v$ is $\frac{i}{4n}2\pi$. The same holds with $L_i$ replaced by $L'_i,R_i$ and $R'_i$.
	\end{enumerate}
\end{lemma}

\begin{proof}
Note that $\{v_0,v_1,\ldots,v_i\}$ are the vertices of $\partial\Pi_{L_i}\cap\partial\Pi$. Thus the part of (1) concerning $L_i$ holds. We can prove the rest of (1), as well as (2), in a similar way. For (3), pick $v_m$ with $0\le m\le i$, then $\angle_{v_m}(L_i,o)=\frac{n-i}{2n}2\pi$. Since $\Delta(v_mL_io)$ is an isosceles triangle with $v_m$ being the apex, (3) follows.
\end{proof}

\begin{lemma}\
	\label{lem:two interiors}
\begin{enumerate}
	\item $L_i$ and $L_j$ (or $R_i$ and $R_j$, $L'_i$ and $L'_j$, $R'_i$ and $R'_j$) are connected by an edge in $\Lambda_v$ if and only if $|j-i|\le n-2$. Moreover, the length of this edge is $\phi(\frac{|j-i|}{2n}2\pi)$ (see Definition~\ref{def:length} for $\phi$).
	\item $L_i$ and $R_j$ (or $L'_i$ and $R'_j$) are connected by an edge in $\Lambda_v$ if and only if $i+j-n\ge 2$. Moreover, the length of this edge is $\phi(\frac{2n-i-j}{2n}2\pi)$.
	\item $L_i$ is not adjacent to any $L'_j$ or $R'_j$. $R_i$ is not adjacent to any $L'_j$ or $R'_j$.
\end{enumerate}
\end{lemma}

Note that claims (1) and (2) concern the length, not the angular length of the edge.

\begin{proof}
We prove (1). Suppose without loss of generality that $i<j$. By Lemma~\ref{cor:connected intersection} (3), the number of edges in $\Pi_{L_i}\cap \Pi_{L_j}$ is $n-(j-i)$. Thus $L_i$ and $L_j$ are adjacent if and only if $n-(j-i)\ge 2$. Now the length formula in (1) follows from Lemma~\ref{lem:overlap and length}. Other parts of (1) can be proved in a similar way. (2) can be deduced in a similar way by noting that the number of edges in $\Pi_{L_i}\cap \Pi_{R_j}$ is $i+j-n$. (3) follows from Lemma~\ref{cor:disjoint}. 
\end{proof}

\begin{corollary}
	\label{cor:triangle inequality2}
The angular lengths of the three sides of each triangle in $\Lambda_v$ satisfy the triangle inequality.
\end{corollary}

\begin{proof}
The case when the triangle contains a real vertex follows from Corollary~\ref{cor:triangle inequality1} (2) (consider the $3$--simplex of $X_n$ spanned by this triangle and $v$). Now we assume the triangle has no real vertices. 

\emph{Case 1:} the three vertices of the triangle are $L_i,L_j$ and $L_k$ with $i<j<k$. By Lemma~\ref{lem:overlap and length}, the length of $\overline{oL_i}$ is $\phi(\frac{n-i}{2n}2\pi)$. By Lemma~\ref{lem:two interiors} (1), the length of $\overline{L_iL_j}$ is $\phi(\frac{j-i}{2n}2\pi)$. Since $\frac{n-j}{2n}2\pi+\frac{j-i}{2n}2\pi=\frac{n-i}{2n}2\pi$, we can arrange $L_i,L_j,L_k,o$ in the unit circle as in Figure~\ref{f:angle} left such that the distance between any two points in $\{L_i,L_j,L_k,o\}$ in the Euclidean plane equal to the length of the edge between them in $X_n$. In particular, $\angle_o(L_i,L_j)+\angle_o(L_j,L_k)=\angle_o(L_i,L_k)$.

\emph{Case 2:} the three vertices of the triangle are $L_i,L_j$ and $R_k$ with $i<j$. By Lemma~\ref{lem:two interiors} (2), $\frac{2n-i-k}{2n}2\pi<\pi$ and the length of $\overline{L_iR_k}$ is $\phi(\frac{2n-i-k}{2n}2\pi)$. Thus we can arrange $L_i,L_j,o,R_k$ as in Figure~\ref{f:angle} right and argue as before. Then other cases are similar.
\end{proof}

\begin{figure}[ht!]
	\centering
	\includegraphics[width=1\textwidth]{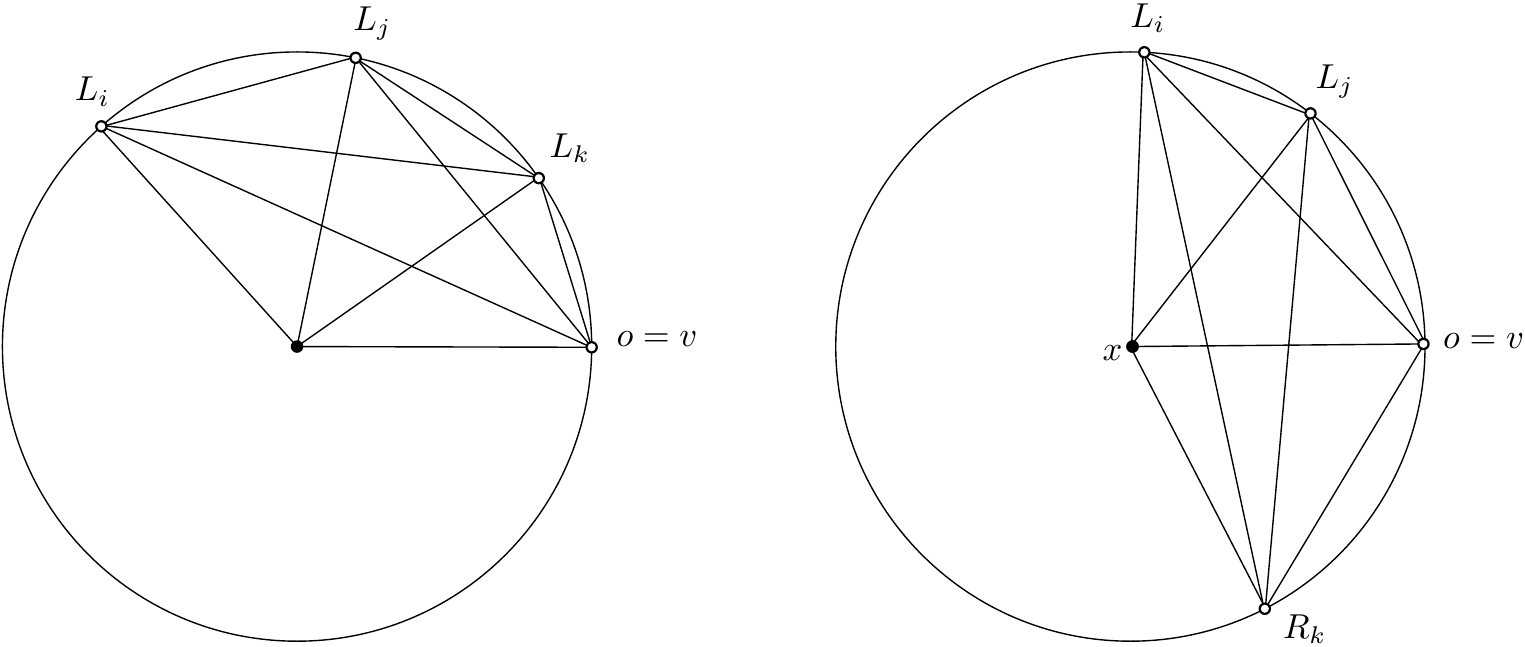}
	\caption{}
	\label{f:angle}
\end{figure}

Let $\Lambda^+_v$ be the full subgraph of $\Lambda_v$ spanned by $$\{v_0,v_1,\ldots,v_n\}\cup \{L_2,L_3,\ldots,L_{n-1}\}\cup \{R_2,R_3,\ldots,R_{n-1}\}.$$ Let $\Lambda^-_v$ be the full subgraph of $\Lambda_v$ spanned by $$\{v'_0,v'_1,\ldots,v'_n\}\cup \{L'_2,L'_3,\ldots,L'_{n-1}\}\cup \{R'_2,R'_3,\ldots,R'_{n-1}\}.$$

\begin{lemma}
	\label{lem: interior tree of cliques}
Each of $\Lambda^+_v$ and $\Lambda^-_v$ is a tree of cliques.
\end{lemma}

\begin{proof}
We only consider $\Lambda^+_v$ since $\Lambda^-_v$ is similar. We define a sequence of collections of vertices of $\Lambda^+_v$ as follows. Let $S_1=\{L_2,\ldots,L_{n-1}\}$, $S_{n-1}=\{R_{n-1},\ldots,R_2\}$, and $S_i=\{L_{i+1},\ldots,L_{n-1},R_{n-1},\ldots,R_{n-i+1}\}$ for $2\le i\le n-2$. By Lemma~\ref{lem:two interiors}, each $S_i$ spans a clique. Moreover, any pair of adjacent interior vertices in $\Lambda_v$ are contained in at least one of the $S_i$. 

For $1\le i\le n$, let $K_i=\{v_i,v_{i-1}\}$. For $1\le i\le 2n-2$, $V_i=K_{ \lfloor i/2 \rfloor+1}\cup S_{\lceil i/2 \rceil}$ (e.g.\ $V_1=K_1\cup S_1$, $V_2=K_2\cup S_1$, $V_3=K_2\cup S_2$, $V_4=K_3\cup S_2$, $\ldots$ , $V_{2n-3}=K_{n-1}\cup S_{n-1}$, $V_{2n-2}=K_{n}\cup S_{n-1}$). By Lemma~\ref{lem:real interior}, each $V_i$ spans a clique $\Delta_i$. Moreover, any edge of $\Lambda^+_v$ is contained in at least one of the $\Delta_i$. Thus $\Lambda^+_v=\cup_{i=1}^{2n-2}\Delta_i$. Note that for each $2\le i\le 2n-2$, $\Delta_{i}\setminus\Delta_{i-1}$ has exactly one vertex, and this vertex is not contained in $\cup_{m=1}^{i-1}\Delta_m$. Thus Definition~\ref{def:tree} (2) holds, hence $\Lambda^+_v$ is a tree of cliques.
\end{proof}

\begin{lemma}
	\label{lem:interior $2$--full cycle contains tips}
Let $\sigma\subset\Lambda_v$ be a simple cycle such that $\sigma\nsubseteq\Lambda^+_v$ and $\sigma\nsubseteq\Lambda^-_v$. Then $v_0\in\sigma$ and $v_n\in\sigma$. Consequently, if $\sigma$ is  a $2$--full simple $n$--cycle in $\Lambda_v$ for $n\ge 4$, then $v_0\in\sigma$ and $v_n\in\sigma$.
\end{lemma}

\begin{proof}
It follows from Lemma~\ref{lem:real interior} and Lemma~\ref{lem:two interiors} (3) that there are no edges between a vertex in $\Lambda^+_v\setminus \{v_0,v_n\}$ and a vertex in $\Lambda^-_v\setminus \{v_0,v_n\}$. Based on Lemma~\ref{lem: interior tree of cliques}, the rest of the proof is identical to the proof of Lemma~\ref{lem:$2$--full cycle contains tips}.
\end{proof}

\begin{lemma}
	\label{lem:pi lower bound1}
	Any edge path in $\Lambda_v$ from $v_0$ to $v_n$ has angular length $\ge \pi$.
\end{lemma}

\begin{proof}
Let $\omega\subset \Lambda_v$ be an edge path from $v_0$ to $v_n$. As in the proof of Lemma~\ref{lem:pi lower bound}, we only consider the case $\omega\subset\Lambda^+_v$.

\emph{Case 1:} There are two adjacent vertices of $\omega$ such that one is $L_j$ and another is $R_k$. Note that $v_0$ is adjacent to $L_j$, and $R_k$ is adjacent to $v_n$. As in the proof of Lemma~\ref{cor:triangle inequality2}, we arrange $L_j,o$ and $R_k$ consecutively in a unit circle such that the Euclidean distance between any of two points in $\{L_j,o,R_k\}$ equals to the length of the edge in $X_n$ between these two points. Consequently $\angle_v(L_j,R_k)=\angle^E_o(L_j,R_k)$, where $\angle^E$ denotes the angle in the Euclidean plane. We refer to Figure~\ref{f:angle} right. Recall that in such an arrangement, $\angle^E_x(L_j,o)=\frac{n-j}{2n}2\pi$, $\angle^E_x(o,R_k)=\frac{n-k}{2n}2\pi$ and $\angle^E_x(L_j,R_k)=\frac{2n-k-j}{2n}2\pi<\pi$. By Lemma~\ref{lem:real interior}, $\angle_v(v_0,L_j)=\frac{j}{4n}2\pi$, thus $\angle_v(v_0,L_j)=\angle^E_{L_j}(x,o)$. Similarly, $\angle_v(R_k,v_n)=\angle^E_{R_k}(x,o)$. Since $\angle^E_{L_j}(x,o)+\angle^E_{R_k}(x,o)+\angle^E_o(L_j,R_k)=2\pi-\angle^E_x(L_j,R_k)>\pi$. Hence $\angle_v(v_0,L_j)+\angle_v(L_j,R_k)+\angle_v(R_k,v_n)>\pi$. By Lemma~\ref{lem:cycle in tree} (2), Corollary~\ref{cor:triangle inequality2} and Lemma~\ref{lem: interior tree of cliques}, we have $\length_\angle(\omega)\ge \angle_v(v_0,L_j)+\angle_v(L_j,R_k)+\angle_v(R_k,v_n)>\pi$.

\emph{Case 2:} Suppose case (1) is not true and $\omega\cap \{L_2,\ldots,L_{n-1}\}\neq\emptyset$. We suppose in addition that after the last vertex of $\omega$ in $\{L_2,\ldots,L_{n-1}\}$ (say $L_i$), $\omega$ still contains at least one vertex from $\{R_2,\ldots,R_{n-1}\}$.

Then $L_i$ is followed by a sub-path $\omega'$ of $\omega$ with $\omega'\subset\partial\Pi$, and then a vertex $R_k$. Suppose the first and the last vertices of $\omega'$ are $v_m$ and $v_{m'}$ respectively. Since $L_i$ and $v_m$ are adjacent, we have $m\le i$ by Lemma~\ref{lem:real interior} (1). Similarly, $m'\ge n-j$. By Lemma~\ref{lem:cycle in tree} (2) and Lemma~\ref{lem: interior tree of cliques}, $\length_\angle(\omega)\ge \angle_v(v_0,L_i)+\angle_v(L_i,v_m)+\length_\angle(\omega')+\angle_v(v_{m'},R_j)+\angle_v(R_j,v_n)$. By Lemma~\ref{lem:real interior},
$$\angle_v(v_0,L_i)=\angle_v(L_i,v_m)=\frac{i}{4n}2\pi;\ \ \ \angle_v(R_j,v_n)=\angle_v(v_{m'},R_j)=\frac{j}{4n}2\pi.$$ 
We are done if $i+j\ge n$. Now we assume $i+j<n$. Then $m\le i<n-j\le m'$ and $\length_\angle(\omega')\ge \frac{n-j-i}{2n}2\pi$. Hence we still have $\length_\angle(\omega)\ge\pi$.

\emph{Case 3:} Suppose case (1) is not true and $\omega\cap \{L_2,\ldots,L_{n-1}\}\neq\emptyset$. We suppose in addition that after the last vertex of $\omega$ in $\{L_2,\ldots,L_{n-1}\}$ (say $L_i$), $\omega$ does not contain any vertex from $\{R_2,\ldots,R_{n-1}\}$.

Then $L_i$ is followed by a sub-path $\omega'$ of $\omega$ traveling from $v_m$ to $v_n$. It follows that $\length_\angle(\omega)\ge \angle_v(v_0,L_i)+\angle_v(L_i,v_m)+\length_\angle(\omega')$. By Lemma~\ref{lem:real interior}, $m\le i$ and  $\angle_v(v_0,L_i)=\angle_v(L_i,v_m)=\frac{i}{4n}2\pi$. It follows that $\length_\angle(\omega')\ge \frac{n-m}{2n}2\pi\ge \frac{n-i}{2n}2\pi$, and hence $\length_\angle(\omega)\ge \pi$.

\emph{Case 4:} Suppose $\omega\cap\{R_2,\ldots,R_{n-1}\}\neq\emptyset$ and $\omega\cap\{L_2,\ldots,L_{n-1}\}=\emptyset$. This is similar to the previous case.

\emph{Case 5:} The remaining case is that $\omega$ does not contain any interior vertices. Then it is clear that $\length_\angle(\omega)\ge \pi$.
\end{proof}

\begin{proof}[Proof of Proposition~\ref{prop:real link} (for interior vertices)]
In view of Corollary~\ref{cor:triangle inequality2}, it suffices to prove any $2$--full simplex $n$--cycle in $\Lambda_v$ with $n\ge 4$ has angular length $\ge 2\pi$. But this follows from Lemma~\ref{lem:interior $2$--full cycle contains tips} and Lemma~\ref{lem:pi lower bound1}.
\end{proof}

The following is an analog of Lemma~\ref{lem:exactly pi} in the case of interior vertex. It will be crucial for applications in \cite{Artinsrigidity}.

\begin{lemma}
	Suppose $v$ is interior and $\omega$ is an edge path in $\Lambda_v$ from $v_0$ to $v_n$ of angular length $\pi$. Then either $\omega\subset\Lambda^+_v$ or $\omega\subset\Lambda^-_v$. If $\omega\subset\Lambda^+_v$, then the following are the only possibilities for $\omega$:
	\begin{enumerate}
		\item $\omega$ does not contain interior vertices, i.e.\ $\omega=v_0\to v_1\to\cdots\to v_n$;
		\item $\omega=v_0\to v_1\to\cdots\to v_{n-i_1}\to R_{i_1}\to\cdots \to R_{i_m}\to v_n$, where $i_1>\cdots>i_m\ge2$;
		\item $\omega=v_0\to L_{i_1}\to\cdots\to L_{i_m}\to v_{i_m}\to v_{i_m+1}\to\cdots \to v_n$, where $2\le i_1<\cdots< i_m$;
		\item $\omega=v_0\to L_{i_1}\to\cdots\to L_{i_m}\to v_{i_m}\to v_{i_{m+1}}\to\cdots \to v_{n-i'_1}\to R_{i'_1}\to\cdots \to R_{i'_{m'}}\to v_n$, where $2\le i_1<\cdots<i_m$, $i'_{1}>\cdots>i'_{m'}\ge 2$, and $i_m\le n-i'_1$.
	\end{enumerate}
	A similar statement holds when $\omega\subset\Lambda^-_v$.
\end{lemma}

\begin{proof}
We argue as Lemma~\ref{lem:exactly pi} to show that $\omega$ is embedded, and that $\omega\subset\Lambda^+_v$ or $\omega\subset\Lambda^-_v$. Now we assume $\omega\subset\Lambda^+_v$.
	
A \emph{left interior component} of $\omega$ is a maximal connected sub-path of $\omega$ such that each of its vertices is one of the $L_i$. We define a \emph{right interior component} in a similar way. By case 1 of Lemma~\ref{lem:pi lower bound1}, there is at least one real vertex between a left interior component and a right interior component.

We show there is at most one left interior component. Suppose the contrary is true. Let $L_i$ be the first vertex of the last left interior component. The vertex in $\omega$ preceding $L_i$ is a real vertex, which we denote by $v_{i_0}$. Since $\omega$ is embedded, $i_0>0$. Let $\omega'$ be the edge path consisting of the edge $\overline{v_0 L_i}$ together with the part of $\omega$ from $L_i$ to $v_n$. By Lemma~\ref{lem:pi lower bound1}, $\length_\angle(\omega')\ge \pi$. Since $\angle_v(v_0,L_i)=\angle_v(v_{i_0},L_i)$ and $v_{i_0}\neq v_0$, $\length_\angle(\omega')<\length_\angle(\omega)=\pi$, which leads to a contradiction.

The same argument also shows that if $L_i\in \omega$ then the vertex of $\omega$ preceding $L_i$ can not be some $v_{i'}$ with $i'\neq 0$. Thus, if there were a left interior component, then the vertex of $\omega$ following $v_0$ would be contained in such component. 

Suppose there is a left interior component. Let $L_i$ be the last vertex in this component and let $v_{i'}$ be any real vertex in the sub-path of $\omega$ from $L_i$ to $v_n$. Then $i'\ge i$. To see this, we suppose the contrary $i'<i$ is true. Let $\omega'$ be the edge path consisting of $\overline{v_0L_{i'}}$, $\overline{L_{i'}v_{i'}}$, and the part of $\omega$ from $v_{i'}$ to $v_n$. By Lemma~\ref{lem:cycle in tree} (2) and Lemma~\ref{lem: interior tree of cliques}, the angular length of the sub-path of $\omega$ from $v_0$ to $L_i$ (from $L_i$ to $v_{i'}$) is $\ge \angle_v(v_0,L_i)$ (resp.\ $\angle_v(L_i,v_{i'})$). By Lemma~\ref{lem:real interior} (3), $\angle_v(v_0,L_i)>\angle_v(v_0,L_{i'})$ and $\angle_v(L_i,v_{i'})>\angle_v(L_{i'},v_{i'})$. Thus $\length_\angle(\omega')<\length_\angle(\omega)=\pi$, which is contradictory to Lemma~\ref{lem:pi lower bound1}.

We claim that if there are two consecutive vertices $L_i$ and $L_j$ in $\omega$ such that $\omega$ reaches $L_i$ first, then $i<j$. To see this, note that by the proof of Corollary~\ref{cor:triangle inequality2} (we can think the center of the circle in Figure~\ref{f:angle} left is $v_0$), $\angle_v(v_0,L_{i'})+\angle_v(L_{i'},L_{j'})=\angle_v(v_0,L_{j'})$ for $i'<j'$. Thus if $i>j$, then the concatenation of $\overline{v_0L_{j}}$ and the sub-path of $\omega$ from $L_j$ to $v_n$ has angular length $<\length_\angle(\omega)=\pi$, which contradicts Lemma~\ref{lem:pi lower bound1}.

We can repeat the above discussion to obtain analogous statements for right interior components. Then the lemma follows.
\end{proof}

\section{The complexes for $2$--dimensional Artin groups}
\label{s:general}
In this section we finalize the proof of one of the main results of the article, namely Theorem~\ref{t:main}
from Introduction. More precisely, for any two-dimensional Artin group $A_{\Gamma}$ we construct a metric simplicial complex $X_\Gamma$, by gluing together complexes $X_n$ for $2$--generated subgroups
of $A_\Gamma$. In Lemma~\ref{lem:sc} we prove that $X_\Gamma$ is simply connected, and in Lemma~\ref{lem:2pi large} we show that links of vertices in $X_\Gamma$ are $2\pi$--large.
As an immediate consequence we obtain the main result of this section:
\begin{theorem}
	\label{t:Xmetrsys}
	$X_\Gamma$ is metrically systolic. Consequently, each $2$--dimensional Artin group is metrically systolic.
\end{theorem}

Let $A_\Gamma$ be an Artin group with defining graph $\Gamma$. Let $\Gamma'\subset\Gamma$ be a full subgraph with induced edge labeling and let $A_{\Gamma'}$ be the Artin group with defining graph $\Gamma'$. Then there is a natural homomorphism $A_{\Gamma'}\to A_{\Gamma}$. By \cite{Van1983homotopy}, this homomorphism is injective. Subgroups of $A_{\Gamma}$ of the form $A_{\Gamma'}$ are called \emph{standard subgroups}.

Let $P_{\Gamma}$ be the standard presentation complex of $A_{\Gamma}$, and let $X^{\ast}_{\Gamma}$ be the universal cover of $P_{\Gamma}$. We orient each edge in $P_{\Gamma}$ and label each edge in $P_{\Gamma}$ by a generator of $A_\Gamma$. Thus edges of $X^{\ast}_{\Gamma}$ have induced orientation and labeling. There is a natural embedding $P_{\Gamma'}\hookrightarrow P_\Gamma$. Since $A_{\Gamma'}\to A_{\Gamma}$ is injective, $P_{\Gamma'}\hookrightarrow P_\Gamma$ lifts to various embeddings $X^{\ast}_{\Gamma'}\to X^{\ast}_{\Gamma}$. Subcomplexes of $X^{\ast}_{\Gamma}$ arising in such way are called \emph{standard subcomplexes}.

A \emph{block} of $X^{\ast}_{\Gamma}$ is a standard subcomplex which comes from an edge in $\Gamma$. This edge is called the \emph{defining edge} of the block. Two blocks with the same defining edge are either disjoint, or identical.

We define precells of $X^{\ast}_{\Gamma}$ as in Section \ref{subsec:precells}, and subdivide each precell as in Figure \ref{f:cell} to obtain a simplicial complex $\Xb_\Gamma$. Interior vertices and real vertices of $\Xb_{\Gamma}$ are defined in a similar way. We record the following simple observations.

\begin{lemma}\
	\label{lem:action}
	\begin{enumerate}
		\item Each element of $A_\Gamma$ maps one block of $\Xb_\Gamma$ to another block with the same defining edge;
		\item if $g\in A_\Gamma$ such that $g$ maps an interior vertex of a block of $\Xb_\Gamma$ to another interior vertex of the same block, then $g$ stabilizes this block;
		\item the stabilizer of each block of $\Xb_\Gamma$ is a conjugate of a standard subgroup of $A_{\Gamma}$.
	\end{enumerate}
\end{lemma}

Within each block of $\Xb_\Gamma$, we add edges between interior vertices as in Section \ref{subsec:subividing and adding new edges}. Then we take the flag completion to obtain $X_\Gamma$. By Lemma~\ref{lem:action}, the newly added edges are compatible with the action of deck transformations $A_{\Gamma}\curvearrowright\Xb_\Gamma$. Thus the action $A_{\Gamma}\curvearrowright\Xb_\Gamma$ extends to a simplicial action $A_{\Gamma}\curvearrowright X_{\Gamma}$, which is proper and cocompact. A \emph{block} in $X_{\Gamma}$ is defined to be the full subcomplex spanned by vertices in a block of $\Xb_\Gamma$. Two blocks of $X_\Gamma$ that have the same defining edge are either disjoint or identical.

\begin{lemma}
	\label{lem:iso}
	Any isomorphism between a block in $X^{\ast}_{\Gamma}$ and the space $\Xa_n$ (cf. Section \ref{subsec:precells}) that preserves the labeling and orientation of edges extends to an isomorphism between a block in $X_{\Gamma}$ and the space $X_n$ (cf. Section \ref{subsec:subividing and adding new edges}).
\end{lemma}

\begin{proof}
	By our construction, it suffices to show that if two vertices $v_1$ and $v_2$ in a block $B\subset X^{\ast}_\Gamma$ are not adjacent in this block, then they are not adjacent in $\Xa_\Gamma$. However, this follows from a result of Charney and Paris (\cite{charney2014convexity}) that $B^{(1)}$ is convex with respect to the path metric on the $1$--skeleton of $X^{\ast}_\Gamma$.
\end{proof}

\begin{lemma}
	\label{lem:sc}
$X_\Gamma$ is simply connected.
\end{lemma}

\begin{proof}
Let $f$ be an edge of $X_{\Gamma}$ not in $\Xb_\Gamma$. Then there are two cells $\Pi_1$ and $\Pi_2$ such that $f$ connects the interior vertices $o_1\in\Pi_1$ and $o_2\in\Pi_2$. By construction, $\Pi_1$ and $\Pi_2$ are in the same block. Thus $f$ and a vertex of $\Pi_1\cap\Pi_2$ span a triangle. By flagness of $X_{\Gamma}$, $f$ is homotopic rel its end points to the concatenation of other two sides of this triangle, which is inside $\Xb_\Gamma$. 

Now we show that each loop in $X_{\Gamma}$ is null-homotopic. Up to homotopy, we assume this loop is a concatenation of edges of $X_{\Gamma}$. If some edges of this loop are not in $\Xb_\Gamma$, then we can homotop these edges rel their end points to paths in $\Xb_\Gamma$ by the previous paragraph. Thus this loop is homotopic to a loop in $\Xb_{\Gamma}$, which must be null-homotopic since $\Xb_\Gamma$ is simply connected. 
\end{proof}

Next, we assign lengths to edges of $X_\Gamma$ in an $A_\Gamma$--invariant way.

Let $B\subset X_\Gamma$ be a block with its defining edge labeled by $n$. By Lemma~\ref{lem:iso}, there is a simplicial isomorphism $i:B\to X_n$ that is label and orientation preserving. Note that all the edges between real vertices of $X_n$ has the same length, which we denote $p$. We define the length of an edge $e\subset B$ to be length$(i(e))/p$. Then, for each vertex $b\in B$, the isomorphism lk$(b,B^{(1)})\to$lk$(i(b),X^{(1)}_n)$ induced by $i$ preserves the angular lengths of edges. In particular, Proposition~\ref{prop:real link} holds for $B$. 

We repeat this process for each block of $X_\Gamma$. Each edge of $X_\Gamma$ belongs to at least one block, so it has been assigned at least one value of length. If an edge belongs to two different blocks, then both endpoints of this edge are real vertices, hence all values of lengths assigned to this edge equal to $1$ by the previous paragraph. In summary, each edge of $X_{\Gamma}$ has a well-defined length. Moreover, such assignment of lengths is $A_\Gamma$--invariant by Lemma~\ref{lem:action}.

\begin{lemma}
	\label{lem:in block}
Each simplex of $X_\Gamma$ is contained in a block.
\end{lemma}

\begin{proof}
Suppose there is an interior vertex $v$ of the simplex $\Delta$. Let $\Pi$ be the cell containing $v$ and $B$ be the unique block containing $v$. Then any real vertex of $X_\Gamma$ adjacent to $v$ is contained in $\Pi$ and any interior vertex of $X_\Gamma$ adjacent to $v$ is contained in $B$. Since $B$ is a full subcomplex, we have $\Delta\subset B$. If $\Delta$ does not contain any interior vertices, then $\Delta$ is a point, or an edge, and the lemma is clear.
\end{proof}

In particular, each triangle of $X_\Gamma$ is contained in a block, its side lengths satisfy the strict triangle inequality by Lemma~\ref{lem:strict triangle inequality}. We define the angular metric on the link of each vertex of $X_\Gamma$ as before.

\begin{lemma}
	\label{lem:2pi large}
Let $v\in X^{(1)}_\Gamma$ be a vertex and let $\Lambda_v=lk(v,X^{(1)}_\Gamma)$.
\begin{enumerate}
	\item The angular lengths of the three sides of each triangle in $\Lambda_v$ satisfy the triangle inequality.
	\item $\Lambda_v$ is $2\pi$--large.
\end{enumerate}
\end{lemma}

\begin{proof}
The $3$--simplex spanned by $v$ and a triangle in $\Lambda_v$ is inside a block by Lemma~\ref{lem:in block}. Then (1) follows from Corollary~\ref{cor:triangle inequality1} and Corollary~\ref{cor:triangle inequality2}.
	
Now we prove (2). If $v$ is an interior vertex, then there is a unique block $B\subset X_{\Gamma}$ containing this vertex, and any other vertex in $X^{(1)}_{\Gamma}$ adjacent to $v$ is contained in this block. Since $B$ is a full subcomplex of $X_{\Gamma}$, lk$(v,X^{(1)}_{\Gamma})=$lk$(v,B^{(1)})$, which is $2\pi$--large by Proposition~\ref{prop:real link}.

We assume $v$ is a real vertex. Let $\omega$ be a $2$--full simple $n$--cycle in lk$(x,X^{(1)}_\Gamma)$ for $n\ge 4$. If $\omega$ is contained in a block, then we know that $\length_\angle(\omega)\ge2\pi$ by Proposition~\ref{prop:real link}. The case when $\omega$ is not contained in a block follows from Lemma~\ref{lem:2pi cycle}.
\end{proof}

\begin{lemma}
	\label{lem:2pi cycle}
	Let $v\in X_\Gamma$ be a real vertex and let $\Lambda_v=lk(v,X^{(1)}_\Gamma)$. Let $\omega$ be a simple cycle with angular length $\le 2\pi$ in the link of $v$. Then exactly one of the following four situations happens:
	\begin{enumerate}
		\item $\omega$ is contained in one block;
		\item $\omega$ travels through two different blocks $B_1$ and $B_2$ such that their defining edges intersect in a vertex $a$, and $\omega$ has angular length $\pi$ inside each block; moreover, there are exactly two vertices in $\omega\cap B_1\cap B_2$ and they correspond to an incoming $a$--edge and an outgoing $a$--edge based at $v$;
		\item $\omega$ travels through three blocks $B_1,B_2,B_3$ such that the defining edges of these blocks form a triangle $\Delta(abc)\subset \Gamma$ and $\frac{1}{n_1}+\frac{1}{n_2}+\frac{1}{n_3}=1$ where $n_1$, $n_2$ and $n_3$ are labels of the edges of this triangle; moreover, $\omega$ is a $6$--cycle with its vertices alternating between real and fake such that the three real vertices in $\omega$ correspond to an $a$--edge, a $b$--edge and a $c$--edge based at $v$;
		\item $\omega$ travels through four blocks such that the defining edges of these blocks form a full $4$--cycle in $\Gamma$; moreover, $\omega$ is a $4$--cycle with one edge of angular length $\pi/2$ in each block.
	\end{enumerate}
Note that in cases (2), (3) and (4), $\omega$ actually has angular length $2\pi$.
\end{lemma}

\begin{proof}
Note that each interior vertex of $\Lambda_v$ is contained in a unique block. Since each edge of $\Lambda_v$ contains at least one interior vertex (otherwise we would have a triangle in $X_\Gamma$ with all its vertices being real, which is impossible), each edge of $\Lambda_v$ is contained in a unique block. Thus, there is a decomposition $\omega=\{\omega_i\}_{i\in \mathbb Z/n\mathbb Z}$ such that
\begin{enumerate}
	\item each $\omega_i$ is a maximal sub-path of $\omega$ that is contained in a block (we denote this block by $B_i$);
	\item $\omega_i\cap \omega_{i+1}$ is made of one or two real vertices.
\end{enumerate}
Let $\{v_i\}_{i\in \mathbb Z/n\mathbb Z}$ be real vertices in $\omega$ such that the endpoints of $\omega_i$ are $v_i$ and $v_{i+1}$. It follows from Lemma~\ref{lem:distance between real vertices} that $\length_\angle(\omega_i)\ge \pi/2$. Thus $n\le 4$. The case $n=4$ leads to case (4) in the lemma. It remains to consider the case $n=3$ and $n=2$.

Each $v_i$ arises from an edge between $x$ and $v_i$. This edge is inside $\Xa_\Gamma$, hence it is labeled by a generator of $A_\Gamma$, corresponding to a vertex $z_i\in \Gamma$. Since $v_i$ corresponds to either an incoming, or an outgoing edge labeled by $z_i$, we will also write $v_i=z^i_i$ or $v_i=z^o_i$. Let $e_i$ be the defining edge of $B_i$. Then 
\begin{equation}
\label{eq:intersection}
z_{i+1}\in e_i\cap e_{i+1}.
\end{equation} 

If $n=2$, then \eqref{eq:intersection} implies that $z_0=z_1$. Thus Lemma~\ref{lem:distance between real vertices} (2) implies that $\length_\angle(\omega_i)\ge \pi$ for $i=0,1$. Thus case (2) in the lemma follows.

Suppose $n=3$. Recall that two blocks of $X_\Gamma$ with the same defining edge are either disjoint or identical. Thus $e_i\neq e_{i+1}$ (otherwise both $\omega_i$ and $\omega_{i+1}$ are contained in $B_i$). By \eqref{eq:intersection}, either $z_0=z_1=z_2$, or $e_0$, $e_1$ and $e_2$ span a triangle $\Delta$ in $\Gamma$. The former case is not possible because of the parity. Let $n_i$ be the label of $e_i$. Note that $z_i\neq z_{i+1}$ for each $i$. Hence $\length_\angle(\omega_i)\ge\frac{n_i-1}{2n_i}2\pi$ by Lemma~\ref{lem:distance between real vertices} (1). Thus $\length_\angle(\omega)\ge (3-\frac{1}{n_0}-\frac{1}{n_1}-\frac{1}{n_2})\pi\ge 2\pi$, where the last inequality follows from the fact that $A_\Gamma$ is $2$--dimensional. Case (3) in the lemma follows.
\end{proof}

\section{Ending remarks and open questions}
\label{s:open}

\subsection{Open questions}
The class of metrically systolic complexes contains the class of all $2$--dimensional $CAT(0)$ piecewise Euclidean simplicial complexes and the class of systolic complexes. This motivates the following natural questions.

\begin{question}
Let $G$ be a metrically systolic group. Is every abelian subgroup of $G$ quasi-isometrically embedded? Are solvable subgroups of $G$ virtually abelian?
\end{question}
At least we know the answer is positive for $2$--dimensional Artin groups. A proof is given in Section~\ref{subsec:abelian}. See \cite[Chapter II.7]{BridsonHaefliger1999} and \cite{OsajdaPrytula} for results in the CAT(0) and systolic settings.

\begin{question}
Let $G$ be a metrically systolic group. Is the centralizer of any infinite order element in $G$ abstractly commensurable with $F_k\times\mathbb Z$? Here $F_k$ is the free group with $k$ generators and $F_0$ denotes the trivial group.
\end{question}

The answer is affirmative for systolic groups \cite{OsajdaPrytula}.
\begin{question}
Are metrically systolic groups semihyperbolic? Biautomatic?
\end{question}

Biautomaticicty for systolic groups has been established by Januszkiewicz-\' Swi\c atkowski \cite{JanuszkiewiczSwiatkowski2006,Swiatkowski2006}.

\begin{question}
Does every finite group acting on a metrically systolic complex fix a point? 	
\end{question}

A fixed point theorem for CAT(0) spaces follows from convexity of the distance function \cite[Chapter II.2]{BridsonHaefliger1999}. A fixed point theorem for systolic complexes has been proven in \cite{ChepoiOsajda}.

\begin{question}
	\label{q:contractibility}
Let $X$ be a metrically systolic complex. Is $X$ contractible? Does $X$ satisfy $S^kFRC$ in the sense of \cite{JanuszkiewiczSwiatkowski2007} for all $k\ge 2$?
\end{question}

It is proved in Section~\ref{s:sfrc} that $X$ has trivial second homotopy group and $X$ is $S^2 FRC$. It is proved in \cite{Chepoi2000,JanuszkiewiczSwiatkowski2007} that the answer to Question~\ref{q:contractibility} is affirmative for systolic complexes.

\begin{question}
Is there a notion of boundary for metrically systolic complexes which generalizes both the $CAT(0)$ case and the systolic case?
\end{question}

See \cite[Chapter II.8]{BridsonHaefliger1999} for the definition of $CAT(0)$ boundaries and \cite{OsajdaPrzytycki} for the definition of boundaries of systolic complexes.

\subsection{Abelian and solvable subgroups}
\label{subsec:abelian}
For each Artin group $A_\Gamma$, Charney and Davis \cite{CharneyDavis} defined an associated \emph{modified Deligne complex} $D_\Gamma$. Now we recall their construction in the $2$--dimensional case. Vertices of $D_\Gamma$ are in one to one correspondence with left cosets of the form $gA_{\Gamma'}$, where $g\in A_\Gamma$ and $\Gamma'$ is either the empty-subgraph of $\Gamma$ (in which case $A_{\Gamma'}$ is the identity subgroup), or a vertex of $\Gamma$, or an edge of $\Gamma$. The \emph{rank} of a vertex $gA_{\Gamma'}$ is the number of vertices in $\Gamma'$. Note that the set $V$ of the vertices has a partial order induced by inclusion of sets. A collection $\{v_i\}_{i=1}^k\subset V$ spans a $(k-1)$--dimensional simplex if $\{v_i\}_{i=1}^k$ form a chain with respect to the partial order. It is clear that $D_\Gamma$ is a $2$--dimensional simplicial complex, and $A_\Gamma$ acts on $D_\Gamma$ without inversions, i.e.\ if an element of $A_\Gamma$ fixes a simplex of $D_\Gamma$, then it fixes the simplex pointwise.

We endow $D_\Gamma$ with a piecewise Euclidean metric such that each triangle $\Delta(g_1,g_2A_s,g_3A_{\overline{st}})$ is a Euclidean triangle with angle $\pi/2$ at $g_2A_s$ and angle $\frac{\pi}{2n}$ at $g_3A_{\overline{st}}$ with $n$ being the label of the edge $\overline{st}$ of $\Gamma$. By \cite[Proposition 4.4.5]{CharneyDavis}, $D_\Gamma$ is $CAT(0)$ with such metric. As being observed in \cite[Lemma 6]{MR2174269}, the action $A_\Gamma\act D_\Gamma$ is semisimple.

\begin{theorem}
	\label{t:absbgrp}
	Let $A_\Gamma$ be a $2$--dimensional Artin group. Then every abelian subgroup of $A_\Gamma$ is quasi-isometrically embedded.
\end{theorem}

\begin{proof}
	Let $A\subset A_\Gamma$ be an abelian subgroup. By \cite[Theorem B]{CharneyDavis} and \cite[Corollary 1.4.2]{CharneyDavis}, the presentation complex of $A_\Gamma$ is a $K(A_\Gamma,1)$ space. Thus $A$ is a free abelian with rank $\le 2$. First we assume $A\cong \mathbb Z$. Since $A_\Gamma\act D_\Gamma$ is semisimple, by \cite[Chapter II.6]{BridsonHaefliger1999}, either $A$ acts by translation on a $CAT(0)$ geodesic line $\ell\subset D_\Gamma$, or $A$ fixes a point $x\in D_\Gamma$. In the former case, we conclude $A$ is quasi-isometrically embedded by noting that any orbit map from $A_\Gamma$ to $D_\Gamma$ is coarsely Lipschitz. In the later case, since $A_\Gamma$ acts on $D_\Gamma$ without inversions, $A$ fixes a vertex in $D_\Gamma$. Thus, up to conjugation, we may assume that $A$ is contained in a standard subgroup $A_{\Gamma'}$ with $\Gamma'$ being a vertex or an edge. Any dihedral Artin group is $CAT(0)$ \cite{BradyMcCammond2000}, so $A$ is quasi-isometrically embedded in $A_{\Gamma'}$ (alternatively, dihedral Artin groups are C(4)-T(4) \cite{Pride1986}, hence they are biautomatic \cite{GerstenShort1991sc}). By \cite[Theorem 1.2]{charney2014convexity}, $A_{\Gamma'}$ is quasi-isometrically embedded in $A_\Gamma$. Hence $A$ is quasi-isometrically embedded in $A_\Gamma$.
	
	Now we assume $A\cong\mathbb Z\oplus\mathbb Z$. By \cite[Theorem II.7.20]{BridsonHaefliger1999}, either there is an $A$--invariant flat plane $P\subset X_\Gamma$ upon which $A$ acts geometrically, or there is an $A$--invariant $CAT(0)$ geodesic line $\ell\subset X_\Gamma$ upon which $A$ acts cocompactly, or $A$ fixes a point. The first and the third case can be handled in a similar way. Now we assume the second case. There is a group homomorphism $h:A\to \mathbb R$ by considering translation length of elements of $A$ along $\ell$. Since $A$ acts on $D_\Gamma$ by cellular isometries, there exists $\epsilon>0$ such that any element of $A$ with nonzero translation length has translation length $>\epsilon$. Hence $h(A)\cong \mathbb Z$. Thus by passing to a finite index subgroup of $A$, we assume $A=\langle a_1,a_2\rangle$ such that $a_1$ fixes a point in $\ell$ and $a_2$ has nonzero translation length. 
	
	Let $p\colon A_\Gamma\to \ell$ be the composition of an orbit map from $A_\Gamma$ to $D_\Gamma$ and the $CAT(0)$ nearest point projection from $D_\Gamma$ to $\ell$. Then there exists $L>0$ such that $p$ is $L$--Lipschitz. Suppose the translation length of $a_2$ is $L'$. For $i=1,2$, suppose $\langle a_i\rangle\to A_\Gamma$ is an $L_i$--bi-Lipschitz embedding. For $k=n_1a_1+n_2a_2\in A$, let $\lVert k\rVert_{\infty}=\max\{|n_1|,|n_2|\}$. Let $\ast$ be the identity element in $A_\Gamma$ and let $d_{A_\Gamma}$ denote the word metric on $A_\Gamma$ with respect to the standard generating set. If $|n_1|\ge 2L_1L_2|n_2|$, then 
	$$d_{A_\Gamma}(k,\ast)\ge \frac{|n_1|}{L_1}-L_2n_2\ge \frac{|n_1|}{2L_1}\ge \frac{1}{2L_1}\lVert k\rVert_{\infty}.$$
	If $|n_1|<2L_1L_2|n_2|$, then 
	\begin{align*}
	d_{A_\Gamma}(k,\ast)&\ge L^{-1} d_{\ell}(p(n_1a_1+n_2a_2),p(\ast))=L^{-1} d_{\ell}(p(n_2a_2),p(\ast))\\
	&\ge L^{-1}L'|n_2|>\frac{L'}{2LL_1L_2}\lVert k\rVert_{\infty}.
	\end{align*}
	Now we conclude from the above two inequalities that $A\to A_\Gamma$ is a quasi-isometric embedding.
\end{proof}

\begin{corollary}
	\label{c:solv}
	Nontrivial solvable subgroups of $2$--dimensional Artin groups are either $\mathbb Z$ or virtually $\mathbb Z^2$.
\end{corollary}
\begin{proof}
	Observe that, by Theorem~\ref{t:absbgrp} $A_{\Gamma}$ is translation discrete, and by \cite[Theorem 3.4]{Conner2000} 
	solvable subgroups of finite cohomological dimension in translation discrete groups are virtually $\mathbb Z^n$. Since, by \cite[Theorem B]{CharneyDavis} and \cite[Corollary 1.4.2]{CharneyDavis}
	the cohomological dimension is in our case at most $2$ we have the assertion.
\end{proof}


\bibliography{mybib}{}
\bibliographystyle{plain}
\end{document}